\documentclass[10pt,reqno]{amsart}

\usepackage{pifont}
\usepackage{mathrsfs}
\usepackage{amscd,color}
\usepackage{amsmath}
\usepackage{latexsym}
\usepackage{amsfonts}
\usepackage{amssymb}
\usepackage{amsthm}
\usepackage{graphicx}
\usepackage{amsmath,amsfonts}
\usepackage{CJK}
\usepackage[linkcolor=blue,citecolor=blue]{hyperref}

\makeatletter
\newcommand{\Extend}[5]{\ext@arrow0099{\arrowfill@#1#2#3}{#4}{#5}}
\makeatother
\let\pa=\partial

\def\eps\epsilon

\def\cA{{\cal A}}

\def\C{\mathop{\bf C\kern 0pt}\nolimits}
\def\DD{\mathop{\bf D\kern 0pt}\nolimits}
\def\K{\mathop{\bf K\kern 0pt}\nolimits}
\def\N{\mathop{\bf N\kern 0pt}\nolimits}
\def\Q{\mathop{\bf Q\kern 0pt}\nolimits}

\newcommand{\beq}{\begin{equation}}
\newcommand{\eeq}{\end{equation}}
\newcommand{\ben}{\begin{eqnarray}}
\newcommand{\een}{\end{eqnarray}}
\newcommand{\beno}{\begin{eqnarray*}}
\newcommand{\eeno}{\end{eqnarray*}}

\def\R{\mathop{\mathbb R\kern 0pt}\nolimits}

\newtheorem{definition}{Definition}[section]

\newtheorem{theorem}{Theorem}[section]
\newtheorem{proposition}[theorem]{Proposition}
\newtheorem{lemma}[theorem]{Lemma}
\newtheorem{remark}{Remark}[section]

\newtheorem{corollary}[theorem]{Corollary}

\theoremstyle{remark}

\begin{document}

 \title[Inhomogeneous generalized Hartree equation]{ \bf The energy-critical inhomogeneous generalized Hartree equation in 3D}

 \author[C. M. Guzm\'an]{Carlos M. Guzm\'an}
\address{Department of Mathematics, UFF, Brazil}
\email{carlos.guz.j@gmail.com}

 \author[C. B. Xu]{Chengbin Xu}
\address{School of Mathematics and Statistics, Qinghai Normal University, Xining, Qinghai 810008, P.R. China. }
\email{xcbsph@163.com}

\maketitle
\begin{center}
 \textit{Dedicated to the memory of Antonio Guzm\'an}
\end{center}

\begin{abstract}
The purpose of this work is to study the $3D$ energy-critical inhomogeneous generalized Hartree equation
$$
i\pa_tu+\Delta u+|x|^{-b}(I_\alpha\ast|\cdot|^{-b}|u|^{p})|u|^{p-2}u=0,\;\ x\in\R^3,
$$
where $p=3+\alpha-2b$. We establish global well-posedness and scattering below the ground state threshold with non-radial initial data in $\dot{H}^1$. To this end, we exploit the decay of the
nonlinearity, which together with the Kenig-Merle roadmap, allows us to treat the non-radial case as the radial case. In this paper are introduced new techniques to overcome the challenges posed by the presence of the potential and the nonlocal nonlinear term of convolution type. In particular, we also show scattering for the classical generalized Hartree equation ($b=0$) assuming radial data. Additionally, in the defocusing case, we show scattering with general data. We believe that the ideas developed here are robust and can be applicable to other types of nonlinear Hartree equations. In the introduction, we discuss some open problems.

\

\end{abstract}


\setcounter{section}{0}\setcounter{equation}{0}
\section{Introduction}

\noindent

We consider the initial value problem (IVP), for the focusing inhomogeneous generalized Hartree equation (which also call inhomogeneous Choquard equation)
 \begin{align}\label{INLC}
 \begin{cases}
   &i\pa_tu+\Delta u+|x|^{-b}(I_\alpha\ast|\cdot|^{-b}|u|^{p})|u|^{p-2}u=0,\ \ \ \  t\in\R,\ x\in\R^3,\\
   &u(0,x)=u_0(x) \in\dot H^1(\R^3)
 \end{cases}
 \end{align}
 where $u:\R\times\R^3\rightarrow\mathbb{C},\ p=3+\alpha-2b$. The inhomogeneous term is $|\cdot|^{-b}$ for some $b>0$. The Riesz-potential is defined on $\R^3$ by
 $$I_\alpha:=\frac{\Gamma(\frac{3-\alpha}{2})}{\Gamma(\frac{\alpha}2)\pi^{\frac{3}2}2^{\alpha}|\cdot|^{3-\alpha}}:
 =\frac{\mathcal{K}}{|\cdot|^{3-\alpha}},\ \ 0<\alpha<3.$$

The nonlinearity in \eqref{INLC} makes the equation a focusing, energy-critical model. In fact, the $\dot H^{1}(\R^3)$ norm is invariant under the standard scaling
$$u_{\lambda}(t,x)=\lambda^{\frac12}u(\lambda^2 t,\lambda x),$$
 and $u_{\lambda}(t,x)$ is also the solution of  the equation \eqref{INLC}.
Moreover,  equation \eqref{INLC} conserves the energy, defined as the sum of the kinetic and potential energies:
$$E(u):=\int_{\R^3}|\nabla u|^2-\frac1{p}(I_\alpha\ast|\cdot|^{-b}|u|^p)|x|^{-b}|u(x)|^{p}dx.$$

\eqref{INLC} is a particular case of a more general inhomogeneous nonlinear generalized Hartree equation of the form
$$i\partial_tu+\Delta u\pm |x|^{-b}(I_\alpha\ast|\cdot|^{-b}|u|^{p})|u|^{p-2}u=0,$$
for $0<\alpha <N$ and $p\geq2$, which models many physical phenomena, where the factor $|x|^{-b}$ represents
some inhomogeneity in the medium\footnote{It is worth mentioning the nonlinearity $- |x|^{-b}(I_\alpha\ast|\cdot|^{-b}|u|^{p})|u|^{p-2}u$ is known as defocusing case.} (see, \cite{Gi}). Some particular cases of this model arise in the mean-field limit of large systems of non-relativistic atoms and molecules, the propagation of electromagnetic waves in plasmas and so on. See \cite{BC,FL,R} for more details.
Very recently, this model has attracted the attention of mathematicians, focusing mainly on the mass-energy intercritical case. The local/global well-posedness was obtained in \cite{AS} via an adapted Gagliardo--Nirenberg type identity. Then, in \cite{SX}, they study the scattering theory following the argument of \cite{DM-2017-PAMS} under the radial condition. Latterly, the second author \cite{Xu} proved the scattering theory for non-radial case inspired by \cite{Mu}. In this work, we consider the energy-critical case. More precisely, we show global well-posedness and scattering for \eqref{INLC} assuming general initial data in $3D$.

\begin{definition}
We say that the solution $u$ of \eqref{INLC} scatters both in time in $\dot{H}^1(\R^3)$, if there exists $u_{\pm}\in \dot{H}^1$ such that
$$
\lim_{t\rightarrow \,\pm \infty}
\|u(t, x) - e^{it\Delta}u_{\pm}\|_{\dot{H}^1}
 = 0,
$$
where $e^{it\Delta}$ denotes the unitary group associated with the linear equation.
\end{definition}

In the limiting case $b=0$ (homogeneous case), the local and global well-posedness in the energy sub-critical case and mass-critical regime were obtained in \cite{Ca,FY}. The authors in \cite{TV} studied the energy scattering of defocusing global solutions for the inter-critical case. For the focusing case, the results were proved in \cite{Sa,Ar}. The critical-energy case has also been investigated, particularly for the $p = 2$ case known as the Hartree equation. Global well-posedness and scattering for the defocusing Hartree equation were discussed in \cite{MXZ3}. In \cite{MXZ5} treated the global existence and scattering versus finite time blow-up of solutions for focusing case and radial setting. Afterward, in \cite{LiM}, the scattering result was extended to the non-radial case. Moreover, for the generalized case with $p > 2$, the well-posedness in $H^{s_c}$ with $0\leq s_c \leq  1$ for small data was considered in \cite{AR}. Hence, a natural question arises regarding the dynamical behavior of solutions in the energy-critical setting, and even more so when considering the inhomogeneous model ($b\neq 0$). This note aims to show the global well-posedness and scattering for \eqref{INLC}, assuming general initial data. We only consider the three-dimensional case, which, we believe to be more interesting (see Remark \ref{op}). As a consequence, we also prove that the solution of \eqref{INLC} scatters for the defocusing setting. Additionally, we obtain scattering for the homogeneous case with radial data.

This work is inspired by recent studies on the inhomogeneous nonlinear Schr\"odinger equation (INLS), that is, $i\partial_tu+\Delta u+|x|^{-b}|u|^{p-2}u=0$, which study scattering assuming non-radial data in the inter-critical and critical cases (see \cite{CFGM}, \cite{DK}, \cite{GM}). We aim to address the challenges posed by the inhomogeneous models that arise from the broken translation symmetry. Previous works have focused on radial solutions to avoid these challenges. However, in this paper, we overcome these obstacles by introducing new elements in the proof. We investigate the scattering theory for \eqref{INLC} assuming non-radial initial data and under the ground state in $\dot{H}^1$. We state the main result.
\begin{theorem}\label{theorem}
 Let $0<b\leq\min\{\frac{1+\alpha}{3},\frac{\alpha}2\}$. Suppose $u_0\in \dot{H}^1(\R^3)$ satisfies
 \begin{align}\label{threshold}
 E(u_0)<E(W)\ \ and\ \ \|u_0\|_{\dot H^1}<\|W\|_{\dot H^1}.
 \end{align}
Then the solution $u$ to \eqref{INLC} is global in time and scatters in $\dot{H}^1(\R^3)$. Here, $W$ denotes the ground state, i.e., the solution to the elliptic equation
$$\Delta W+(I_\alpha\ast|\cdot|^{-b}|W|^p)|x|^{-b}|W|^{p-2}W=0.$$

\end{theorem}

To this end, we use the robust technique introduced by Kenig--Merle to study the energy-critical nonlinear Schr\"odinger equation (NLS) with new ingredients to show the non-radial case. More precisely, we combined this approach (Kenig--Merle roadmap) together with the decaying factor in the nonlinearity. We reduce the problem of scattering to the problem of precluding the possibility of non-scattering solutions below the ground state threshold that possess certain compactness properties. The proof is by contradiction. That is, we show that there exists minimal energy blow-up solution (also called critical solution) to \eqref{INLC} under the assumption that Theorem \ref{theorem} fails via linear profile and a new nonlinear profile, which is the main new ingredient of the present work (see Proposition \ref{N-profile}). In other words, if Theorem \ref{theorem} fails, then there exists a critical solution $u_c$ of \ref{INLC} with initial data $u_c(0)$ such that
$$
\|u_c(0)\|_{\dot{H}^1}<\|W\|_{\dot{H}^1},\quad E(u_c)=E_c,\quad\textnormal{and}\quad \|u_c\|_{S(I_{\max})}=\infty,
$$
where $E_c$ is a critical value (see Proposition \ref{MEBS}). Moreover, There exists a frequency scale function $N:I_{\max}\to(0,\infty)$ such that
$$\{N(t)^{-\frac12}u(t,N(t)^{-1}x):t\in I_{\max}\}$$
is precompact in $\dot{H}^1(\R^3)$ (see Proposition \ref{precompact}).
\begin{remark}
The restriction $b\leq \min\{\frac{1+\alpha}{3},\frac{\alpha}2\}$ is a very important result to study the well-posedness and stability theory of \eqref{INLC}. It is worth mentioning that the treatment of the Hartree case is different from the nonlinear Schr\"odinger case. In particular, the nonlocal nonlinear term of convolution type, combined with the potential $|x|^{-b}$, presents several challenges in proving the technical results required to use the Kenig-Merle roadmap. Moreover, the new challenge addressed in this paper is the non-radial setting, which has recently become more interesting to the authors.
\end{remark}

\begin{remark}
It is important to highlight that Dodson and Murphy have introduced a simpler approach to prove scattering for general data. This approach combines Tao's scattering criterion, Virial-Morawetz-type estimates, and the decay of the nonlinearity. However, it is primarily effective for inter-critical cases, see \cite{Mu,Xu}. In the case of critical energy, Dodson and Murphy's ideas appear to be insufficient as they require the initial data, $u_0\in L^2$, which is not the case here. Therefore, we turn to the approach provided by the Kenig and Merle roadmap.    
\end{remark}

Here, we focus on the three-dimensional case, however we believe that much of the analysis can be extended to the more general cases, $N\geq 4$. Furthermore, compared to the homogeneous setting ($b=0$), the scattering result for dimensions $N = 3,4$, as well as for non-radial data, are still open problems. As a consequence of Theorem \ref{theorem}, we can conclude that the scattering result holds for dimensions $N = 3$, but only under the assumption of initial radial data. That is,
\begin{corollary}\label{Homog. setting} Consider the Cauchy problem associated with the classical generalized Hartree equation (\eqref{INLC} with $b=0$). Suppose radial $u_0\in \dot{H}^1(\R^3)$ such that
 \begin{align}\label{threshold hom}
 E(u_0)<E(W)\ \ and\ \ \|u_0\|_{\dot H^1}<\|W\|_{\dot H^1}.
 \end{align}
Then the solution $u$ is global in time and scatters both in time.
\end{corollary}

The proof follows mainly from the ideas introduced in the main new ingredient of this work (see Proposition 4.1). We explain in more detail. For the classical nonlinear Schr\"odinger model (NLS), the difference between the radial and non-radial cases is that the minimal non-scattering solutions are parametrized in part by a spatial center $x(t)$, which are ruled out soliton-type solutions as long as one has some control over the size of $x(t)$. The case of radial solutions leads to the best-case scenario $x(t) \equiv 0$. For the inhomogeneous model, we have the same situation (but without radial data), that is, due to the decaying factor in the nonlinearity, a solution with $|x(t)| \rightarrow \infty$ should behave like an approximate solution to the linear Schr\"odinger equation and so decay and scatter. Thus, we obtain minimal non-scattering solutions that must obey $x(t) \equiv 0$.

 Finally, we obtain the scattering result for the defocusing case. It also can be seen as a consequence of Theorem \ref{theorem}. Here we remove the conditions \eqref{threshold} because the norm $\|\cdot\|_{\dot{H}^1}$ is always bounded and all the energies are positive and comparable with the kinetic energy, it is the easier case.

 \begin{corollary}\label{theorem defo}\label{def case}
 Let $0<b\leq\min\{\frac{1+\alpha}{3},\frac{\alpha}2\}$. For any $u_0\in \dot{H}^1(\R^3)$, there exists a unique global solution to
 $$
i\pa_tu+\Delta u-|x|^{-b}(I_\alpha\ast|\cdot|^{-b}|u|^{p})|u|^{p-2}u=0,
$$
with initial data $u_0$. Furthermore, this global solution scatters both in time.
\end{corollary}

To do that, we make small changes in the proof of Theorem \ref{theorem}.

\begin{remark}[{\bf Some open problems}]\label{op}

We conclude the introduction by discussing some open problems. The first open problem is to prove global well-posedness and scattering to \eqref{INLC} for higher dimensions ($N\geq 4$) in both the focusing and defocusing cases.\\
\indent In the particular case, $b=0$, the scattering issue is also an open problem for non-radial data when $N\geq 3$. In this article, we show scattering in $3D$ assuming initial radial data. We believe that for dimensions $N\geq 5$, the result can be obtained using the ideas introduced in the NLS model \cite{KV1}. However, when $N=\{3,4\}$, the scattering result for non-radial data is expected to be more complicated, similar to what happens in the classical Schrödinger equation. By the way, the $3D$ case in the NLS setting is a famous problem in the area, which is why we chose to focus on it in this article. We observe that for the inhomogeneous model, we were able to prove the scattering result in three dimensions, while the homogeneous case is still an open problem.

\end{remark}
This article is organized as follows:
In Section 2, We introduce some notations and results related to the local and stability theory for \eqref{INLC}. In Section $3$, we study the linear profile decomposition as well as establish some results related to the variational characterization of the ground state.
In Section 4, we prove that there exists minimal energy blow-up solution to \eqref{INLC} under the assumption that Theorem \ref{theorem} fails, via linear profile and a new nonlinear profile which is the main ingredient of the work (Proposition \ref{N-profile}). Finally, in Section $5$, We prove the compactness of the minimal energy blow-up solution and preclude the possibility of compact sub-threshold solutions by refined Duhamel formula and Morawetz estimate which implies assumption does not hold, hence Theorem \ref{theorem} holds. Moreover, we show the corollaries.

{\bf Acknowledgments.} C.M.G. was partially supported by Conselho Nacional de Desenvolvimento Científico e Tecnologico -
CNPq and Fundação de Amparo à Pesquisa do Estado do Rio de Janeiro - FAPERJ (Brazil). C.M.G would like to dedicate this paper to the memory of his father, Antonio Guzm\'an Rodríguez, and thank him for his unconditional love and support.

\section{Preliminaries}

\noindent

Let us start this section by giving some notations which
will be used throughout this paper. Moreover, we study the local well-posedness as well as the stability result.

We use $X\lesssim Y$ to denote $X\leq CY$ for some constant $C>0$. Similarly, $X\lesssim_{u} Y$ indicates there exists a constant $C:=C(u)$ depending on $u$ such that $X\leq C(u)Y$.
We also use the big-oh notation $\mathcal{O}$. e.g. $A=\mathcal{O}(B)$ indicates $A\leq CB$ for constant $C>0$.
The derivative operator $\nabla$ refers to the spatial variable only.
We use $L^r(\mathbb{R}^N)$ to denote the Banach space of functions $f:\mathbb{R}^3\rightarrow\mathbb{C}$ whose norm
$$\|f\|_r:=\|f\|_{L^r}=\Big(\int_{\mathbb{R}^3}|f(x)|^r dx\Big)^{\frac1r}$$
is finite, with the usual modifications when $r=\infty$. For any non-negative integer $k$,
we denote by $H^{k,r}(\mathbb{R}^3)$ the Sobolev space defined as the closure of smooth compactly supported functions in the norm $\|f\|_{H^{k,r}}=\sum_{|\alpha|\leq k}\|\frac{\partial^{\alpha}f}{\partial x^{\alpha}}\|_r$, and we denote it by $H^k$ when $r=2$.
For a time slab $I$, we use $L_t^q(I;L_x^r(\mathbb{R}^3))$ to denote the space-time norm
\begin{align*}
  \|f\|_{L_{t}^qL^r_x(I\times \R^3)}=\bigg(\int_{I}\|f(t,x)\|_{L^r_x}^q dt\bigg)^\frac{1}{q}
\end{align*}
with the usual modifications when $q$ or $r$ is infinite, sometimes we use $\|f\|_{L^q(I;L^r)}$ or $\|f\|_{L^qL^r(I\times\mathbb{R}^3)}$ for short.

We also recall Bernstein's inequality and some Strichartz estimates associated with the linear Schr\"odinger propagator. Let $ \psi(\xi) $ be a radial smooth function supported in the ball $ \{ \xi \in \R^3: \vert \xi \vert \le \frac{11}{10} \} $ and equal to 1 on the ball $ \{ \xi \in \R^3 : \vert \xi \vert \le 1 \} $.
For each number $ N > 0 $, we define the Fourier multipliers
\begin{equation*}
\begin{aligned}
\widehat{P_{\le N} g} (\xi) & ~ := \psi \left( \frac{\xi}{N} \right) \hat{g}(\xi), \\
\widehat{P_{> N} g} (\xi) & ~ := \left( 1 - \psi \left( \frac{\xi}{N} \right) \right) \hat{g}(\xi), \\
\widehat{P_N g} (\xi) & ~ := \left( \psi \left( \frac{\xi}{N} \right) - \psi \left( \frac{2\xi}{N} \right) \right) \hat{g}(\xi),
\end{aligned}
\end{equation*}
and similarly $ P_{< N} $ and $ P_{\ge N} $. We also define
\[
P_{M < \cdot \le N} := P_{\le N} - P_{\le M} = \sum_{M < N^\prime \le N} P_{N^\prime},
\]
whenever $ M < N $.
We usually use these multipliers when $M$ and $N$ are dyadic numbers. It is worth saying that all summations over $M$ or $N$ are understood to be over dyadic numbers. Like all Fourier multipliers, the Littlewood--Paley operators commute with the Sch\"odinger flow $e^{it\Delta}$.

As for some applications of the Littlewood--Paley theory, we have the following lemma.
\begin{lemma}[Bernstein inequalities]\label{Bern}
	For any $ 1 \le p \le q \le \infty $ and $ s \ge 0 $,
	\begin{equation*}
	\begin{aligned}
	\Vert P_{\ge N} g \Vert_{L_x^p(\R^3)} \lesssim & ~ N^{-s} \Vert \vert \nabla \vert^s P_{\ge N} g \Vert_{L_x^p(\R^3)} \lesssim N^{-s} \Vert \Vert \nabla \vert^s g \Vert_{L_x^p(\R^3)}, \\
	\Vert \vert \nabla \vert^s P_{\le N} g \Vert_{L_x^p(\R^3)} \lesssim & ~ N^s \Vert P_{\le N} g \Vert_{L_x^p(\R^3)} \lesssim N^s \Vert g \Vert_{L_x^p(\R^3)}, \\
	\Vert \vert \nabla \vert^{\pm s} P_N g \Vert_{L_x^p(\R^6)} \lesssim & ~ N^{\pm s} \Vert P_N g \Vert_{L_x^p(\R^3)} \lesssim N^{\pm s} \Vert g \Vert_{L_x^p(\R^3)} ,\\
	\Vert P_N g \Vert_{L_x^q(\R^3)} \lesssim & ~ N^{\frac{3}{p} - \frac{3}{q}} \Vert P_N g \Vert_{L_x^p(\R^3)},
	\end{aligned}
	\end{equation*}
	where $ \vert \nabla \vert^s $ is the classical fractional-order operator.
\end{lemma}

 We define the set $\Lambda_0$:
$$\Lambda_0:=\left\{(q,r)\ \text{is}\ L^2\text{-admissible};2\leq r\leq 6\right\}.$$
Next, we define the following Strichartz norm
$$\|u\|_{S(L^2,I)}=\sup_{(q,r)\in \Lambda_0}\|u\|_{L_t^qL_x^r(I)}$$
and dual Strichartz norm
$$\|u\|_{N(I)}=\inf_{(q,r)\in \Lambda_{0}}\|u\|_{L_t^{q^{'}}L_x^{r^{'}}(I)}.$$
 If $I=\R$, $I$ is omitted usually.

 Now, let us recall some results about Strichartz estimates and Hardy-Littlewood-Sobolev's inequality, see \cite{Ca}, \cite{KT}, \cite{Li}.
\begin{lemma}\label{Strichartz}
  Let $0\in I$, the following statement hold

  (i)(linear estimate)
  $$\|e^{it\Delta}f\|_{S(L^2)}\leq C\|f\|_{L^2};$$

   (ii)(inhomogeneous estimate)
  $$\left\|\int_{0}^{t}e^{i(t-s)\Delta}g(\cdot,s)ds\right\|_{S(L^2,I)}\leq C\|g\|_{N(I)}.$$

\end{lemma}
\begin{lemma}\label{HLS}
 Let $0<\alpha<N$ and $1<r,s<\infty$ and $(f,g)\in L^p \times L^q$.
\begin{itemize}
\item [(i)] If $\frac1p+\frac1q=1+\frac{\lambda}N$, then
  $$\int_{\R^N}(I_\alpha \ast f)(x) g(x)dx\lesssim \|f\|_{L^p}\|g\|_{L^q}.$$ 
\item [(ii)] If $\frac1p+\frac1q+\frac1r=1+\frac{\lambda}N$, then
  $$
  \|(I_\alpha \ast f) g\|_{L^{r'}}\lesssim \|f\|_{L^p}\|g\|_{L^q}.
  $$ 

\end{itemize}
 \end{lemma}

We end this section with the local well-posedness and stability results, which be very important in our work. For a slab $I\subset\R$, we define the mixed spaces
$$S^1(I):=L_t^{q_0}(I,L_x^{r_1}):=L^{2p}(I,L^{\frac{6p}{p-2}})\ \ 
\quad \textnormal{and}\quad W(I):=L^{q_0}(I,L^{r_0}):=L^{2p}(I,L^{\frac{6p}{3p-2}}),$$
so,
$$\|u\|_{S^1(I)}\leq C\|\nabla u\|_{W(I)}.$$
Let $S(I)$ denote by
$$\|u\|_{S(I)}:=\|u\|_{S^1(I)}+\|\nabla u\|_{W(I)}.$$

Before starting the local theory and the stability results, we show the estimate of the nonlinear term, which plays an important role in proving the mentioned results as well as the scattering theory.
\begin{lemma}[Nonlinear estimate]\label{Non-e1}
  Let $0<b\leq\frac{(1+\alpha)}{3}$. Then, we have
  $$\|\nabla[(I_{\alpha}\ast|\cdot|^{-b}|u|^p)|x|^{-b}|u|^{p-2}u]\|_{N(I)}\lesssim \|u\|_{S^1(I)}^{2(p-1)-2b}\|\nabla u\|_{W(I)}^{1+2b}.$$
\end{lemma}
\begin{proof}
By the definition of $N(I)$,
 \begin{align*}
   \|\nabla[(I_{\alpha}\ast|\cdot|^{-b}|u|^p)|x|^{-b}|u|^{p-2}u]\|_{N(I)}\leq \|\nabla[(I_{\alpha}\ast|\cdot|^{-b}|u|^p)|x|^{-b}|u|^{p-2}u]\|_{L_t^{q_0'}L_x^{r_0'}}.
 \end{align*}
 Note that
\begin{align}\label{Non-1}
|\nabla[(I_{\alpha}\ast|\cdot|^{-b}|u|^p)|x|^{-b}|u|^{p-2}u]|\lesssim& (I_{\alpha}\ast|\cdot|^{-b}|u|^p)|x|^{-b}|u|^{p-2}|\nabla u|\\\label{Non-2}
&+(I_{\alpha}\ast|\cdot|^{-b}|u|^p)|x|^{-b-1}|u|^{p-1}\\
&+(I_{\alpha}\ast|\cdot|^{-b}|u|^{p-1}|\nabla u|)|x|^{-b}|u|^{p-1}\\
&+(I_{\alpha}\ast|\cdot|^{-b-1}|u|^p)|x|^{-b}|u|^{p-1}.
\end{align}
We estimate the first two because the other terms are similar. So, applying Hardy--Littlewood--Sobolev's inequality (Lemma \ref{HLS}) together with H\"older's inequality and the relation
  $$1+\frac{\alpha}{3}=\frac{1}{r_0}+\frac{2(p-1)-2b}{r_1}+\frac{1+2b}{r_0},$$
  it follows that
   \begin{align*}
    \|\eqref{Non-1}\|_{L_t^{q_0'}L_x^{r_0'}}\lesssim&
    \|(I_{\alpha}\ast(|\cdot|^{-1}|u|)^{b}|u|^{p-b})(|x|^{-1}u)^{b}|u|^{p-b-2}\nabla u\|_{L_t^{q_0'}L_x^{r_0'}}\\
    \lesssim&\||x|^{-1}u\|_{L_t^{q_0}L_x^{r_0}}^{b}\|u\|_{L_t^{q_0}L_x^{r_1}}^{p-b}\||x|^{-1}u\|_{L_t^{q_0}L_x^{r_0}}^{b}
    \|u\|_{L_t^{q_0}L_x^{r_1}}^{p-b-2}\|\nabla u\|_{L_t^{q_0}L_x^{r_0}}\\
    \lesssim&\||x|^{-1}u\|_{L_t^{q_0}L_x^{r_0}}^{2b}\|u\|_{L_t^{q_0}L_x^{r_1}}^{2(p-1)-2b}\|\nabla u\|_{L_t^{q_0}L_x^{r_0}},
  \end{align*}
  and
  \begin{align*}
    \|\eqref{Non-2}\|_{L_t^{q_0'}L_x^{r_0'}}\lesssim&
    \|(I_{\alpha}\ast(|\cdot|^{-1}|u|)^{b}|u|^{p-b})(|x|^{-1}u)^{b+1}|u|^{p-b-2}\|_{L_t^{q_0'}L_x^{r_0'}}\\
    \lesssim&\||x|^{-1}u\|_{L_t^{q_0}L_x^{r_0}}^{b}\|u\|_{L_t^{q_0}L_x^{r_1}}^{p-b}\||x|^{-1}u\|_{L_t^{q_0}L_x^{r_0}}^{b+1}
    \|u\|_{L_t^{q_0}L_x^{r_1}}^{p-b-2}\\
    \lesssim&\||x|^{-1}u\|_{L_t^{q_0}L_x^{r_0}}^{2b+1}\|u\|_{L_t^{q_0}L_x^{r_1}}^{2(p-1)-2b},
  \end{align*}
  where we need $p-b-2\geq0$ which implies $b\leq \frac{1+\alpha}3$.

On the other hand, by the Hardy inequality, that is,
$$\||x|^{-1}u\|_{L^r}\lesssim \|\nabla u\|_{L^r}\quad 1<r<3,$$
we get (using the fact $p>2$)
$$\|\eqref{Non-1}\|_{L_t^{q_0'}L_x^{r_0'}}+\|\eqref{Non-2}\|_{L_t^{q_0'}L_x^{r_0'}}\lesssim \|u\|_{S^1(I)}^{2(p-1)-2b}\|\nabla u\|_{W(I)}^{1+2b}.$$
Thus, we complete the proof.
\end{proof}

The next two propositions show results for the Cauchy problem \eqref{INLC}. We begin with the local well-posedness in $\dot{H}^1(\R^3)$.
\begin{proposition}[Local well-posedness]\label{LWP}
  For any $u_0\in \dot{H}^1(\R^3)$  and $0<b\leq \min\{\frac{1+\alpha}{3},\frac{\alpha}2\}$, there exists $T=T(u_0)$ and a unique solution $u$ with initial data $u_0$ satisfying
  $$u\in L_{loc}^q((-T,T), \dot H^{1,r})\ \ \forall (q,r)\in \Lambda_0.$$
  In particular, there exists $\eta_0>0$ so that if
  \begin{align*}
    \|e^{it\Delta}u_0\|_{S^1([0,\infty))}<\eta
  \end{align*}
  for any $\eta<\eta_0$, the solution to \eqref{INLC} is forward global and obeys
  $$\|u\|_{S^1([0,\infty)}\lesssim \eta.$$
  The analogous statement holds backward in time or on all of $\R$.
\end{proposition}
\begin{proof}
  To do that we use the contraction mapping argument. Fix $\rho, r>0$, to be chosen later, we define a complete metric space $(B_{\rho,r}, d)$ as following:
  $$
B_{\rho,r}=\{u\in C(I;\dot{H}^1(\mathbb{R}^3))\;:\; 
 \|u\|_{S^1(I)}\leq \rho \quad \textnormal{and} \quad \|\nabla u\|_{W(I)}\leq r\},
$$
  $$d(u,v):=\|\nabla (u-v)\|_{W(I)}+\| u-v\|_{S^1(I)}.$$
Let $\Gamma (u)$ denoted by
\begin{equation}\label{OPERATOR} 
\Gamma(u)(t)=e^{it\Delta}u_0+i\int_0^t e^{i(t-s)\Delta}|x|^{-b}(I_\alpha\ast|\cdot|^{-b}|u(s)|^{p})|u(s)|^{p-2}u(s)ds.
\end{equation}
We will next prove $\Gamma$ is a contraction on $(B_{\rho,r}, d)$.
 First, by the Strichartz estimates and Lemma \ref{Non-e1}, one can obtain that
\begin{align*}
  \|\nabla\Gamma(u)\|_{W(I)}\leq& C\|u_0\|_{\dot H^1}+C\|u\|_{S^1(I)}^{2(p-1)-2b}\|\nabla u\|_{W(I)}^{1+2b}\\
  \leq&C(\|u_0\|_{\dot H^1}+\rho^{2(p-1)-2b}r^{1+2b})
\end{align*}
and
\begin{align*}
  \|\Gamma(u)\|_{S^1(I)}\leq& \|e^{it\Delta}u_0\|_{S^1(I)}+C\|u\|_{S^1(I)}^{2(p-1)-2b}\|\nabla u\|_{W(I)}^{1+2b}\\
  \leq&\|e^{it\Delta}u_0\|_{S^1(I)}+C\rho^{2(p-1)-2b}r^{1+2b}.
\end{align*}
Choosing $r= 2C\|u_0\|_{\dot H^1}$ and $\rho$ so that $C\rho^{2(p-1)-2b}r^{2b}<\frac{1}{4}$ one has $\|\nabla \Gamma(u)\|_{S(L^2;I)}\leq r$.
Let $\delta=\|e^{it\Delta}u_0\|_{S^1(I)}$, we can choose time interval  $0\in I$ such that $\delta=\frac{\rho}{2}$ and $ C\rho^{2(p-1)-2b-1}r^{1+2b}<\frac14$, then $\|\Gamma(u)\|_{S^1(I)}\leq \rho$ which means $\Gamma(u)\in B_{\rho,r}$.

To complete the proof we show that $\Gamma$ is a contraction on $B_{\rho,r}$. Using Strichartz estimate, we have
\begin{equation}\label{Contrction1}
d(\Gamma(u),\Gamma(v))\leq 2C\|\nabla \left(F(x,u)- F(x,v)\right)\|_{N(I)},
\end{equation}
where $F(u)=(I_\alpha\ast|\cdot|^{-b}|u|^p)|x|^{-b}|u|^{p-2}u$. A direct calculation shows that
 \begin{align*}
  &|\nabla(F(u)-F(v)|\\
  \lesssim& \left|(I_{\alpha}\ast|\cdot|^{-b}|u|^p)|x|^{-b}|u|^{p-2}\nabla u-(I_{\alpha}\ast|\cdot|^{-b}|v|^p)|x|^{-b}|v|^{p-2}\nabla v\right|\\
  &+\left|(I_{\alpha}\ast|\cdot|^{-b}|u|^p)|x|^{-b-1}|u|^{p-2}u- (I_{\alpha}\ast|\cdot|^{-b}|v|^p)|x|^{-b-1}|v|^{p-2}v\right|\\
  &+\big|(I_{\alpha}\ast|\cdot|^{-b}|u|^{p-2}\Re\overline{u}\nabla u)|x|^{-b}|u|^{p-2}u-(I_{\alpha}\ast|\cdot|^{-b}|v|^{p-2}\Re \overline{v} \nabla v)|x|^{-b}|v|^{p-2}v\big|\\
  &+\left|(I_{\alpha}\ast|\cdot|^{-b-1}|u|^p)|x|^{-b}|u|^{p-2}u-(I_{\alpha}\ast|\cdot|^{-b-1}|v|^p)|x|^{-b}|v|^{p-2}v\right|\\
  =&:I+II+III+IV.
 \end{align*}
 We begin with estimating the term $I$ which can be written
  \begin{align*}
    |I|\leq& C\left|(I_{\alpha}\ast|\cdot|^{-b}(|u|^p-|v|^p)\right||x|^{-b}|v|^{p-2}|\nabla v|\\
    &+(I_{\alpha}\ast|\cdot|^{-b}|u|^p)|x|^{-b}\left|(|u|^{p-2}-|v|^{p-2})\right||\nabla v|\\
    &+(I_{\alpha}\ast|\cdot|^{-b}|u|^p)|x|^{-b}|u|^{p-2}|\nabla(u-v)|.
  \end{align*}
So, applying the inequality (using the fact $\sigma\geq 1$):
$$
||f|^{\sigma}-|g|^{\sigma}|\leq C(|f|^{\sigma-1}+|g|^{\sigma-1})|f-g|,
$$
one has (using $\sigma =p-2\geq 1$, that is $p\geq 3$, which is equivalent to $b\leq \frac{\alpha}{2}$)
\begin{align*}
|I|\leq& C\left|(I_{\alpha}\ast|\cdot|^{-b}(|u|^{p-1}+|v|^{p-1})|u-v|\right||x|^{-b}|v|^{p-2}|\nabla v|\\
&+(I_{\alpha}\ast|\cdot|^{-b}|u|^p)|x|^{-b}\left|(|u|^{p-3}+|v|^{p-3})(u-v)\right||\nabla v|\\
&+(I_{\alpha}\ast|\cdot|^{-b}|u|^p)|x|^{-b}|u|^{p-2}|\nabla(u-v)|.
  \end{align*}
The Hardy--Littlewood--Sobolev inequality (Lemma \ref{HLS}) and Lemma \ref{Non-e1} imply
\begin{align*}
  \|I\|_{L_t^{q_0'}L_x^{r_0'}}
\leq& C\Big((\|\nabla u\|_{W(I)}^b+\|\nabla v\|_{W(I)}^b)(\|u\|_{S^1(I)}^{p-1-b}+\|v\|_{S^1(I)}^{p-1-b})\|v\|_{S^1(I)}^{p-2-b}\|\nabla v\|_{W(I)}^{1+b}\|u-v\|_{S^1(I)}\\
  &+\|u\|_{S^1(I)}^{p-b}\|\nabla u\|_{W(I)}^b\|\nabla v\|_{W(I)}\\
  &\ \ \ \ \times
  \begin{cases}
    (\|u\|_{S^1(I)}^{p-3}+\|v\|_{S^1(I)}^{p-3})\|\nabla(u-v)\|_{W(I)}^b\|u-v\|_{S^1(I)}^{1-b},\ \text{ if $b<1;$}\\
    (\|u\|_{S^1(I)}^{p-2-b}\|\nabla u\|_{W(I)}^{b-1}+\|v\|_{S^1(I)}^{p-2-b}\|\nabla v\|_{W(I)}^{b-1})\|\nabla(u-v)\|_{W(I)},\ \text{ if $b\geq1;$}
  \end{cases}\\
  &+\|u\|_{S^1(I)}^{2(p-1)-2b}\|\nabla u\|_{W(I)}^{2b}\|\nabla(u-v)\|_{W(I)}\Big).
\end{align*}
Therefore, if $u, v \in B_{\rho,r}$, then
\begin{align*}
  \|I\|_{L_t^{q_0'}L_x^{r_0'}}\leq C\Big(\rho^{2p-3-2b}r^{1+2b}+\max\{\rho^{2p-3-b}r^{1+b},\rho^{2p-2-2b}r^{2b}\}+\rho^{2p-2-2b}r^{2b}\Big)d(u,v).
\end{align*}
We can choose $\rho$ small enough so that
$$C\Big(\rho^{2p-3-2b}r^{1+2b}+\max\{\rho^{2p-3-b}r^{1+b},\rho^{2p-2-2b}r^{2b}\}+\rho^{2p-2-2b}r^{2b}\Big)\leq \frac18.$$
Thus the other term can be estimated similarly. Hence, we can get
$$d(\Gamma(u),\Gamma(v))\leq \frac12 d(u,v)$$
which means that  $\Gamma$ is a contraction. So, by the contraction mapping principle, $\Gamma$ has a unique fixed point $u\in B_{\rho,r}$.
\end{proof}

Now, we study the stability theory for \eqref{INLC}.
\begin{proposition}[Stability]\label{Stability}
  Let $0<b\leq \min\{\frac{1+\alpha}3,\frac{\alpha}2\}$. Suppose $\tilde{u}:I\times\R^3\rightarrow\mathbb{C}$ obeys
  $$\|\tilde{u}\|_{L_t^\infty\dot{H}_x^1}+\|\tilde{u}\|_{S^1(I)}\leq E<\infty.$$
  Then there exists $\epsilon_1=\epsilon_1(E)>0$ such that if
  \begin{align*}
    \|\nabla&\{(i\partial_t+\Delta)\tilde{u}+|x|^{-b}(I_{\alpha}\ast|\cdot|^{-b}|\tilde{u}|^p)|\tilde{u}|^{p-2}\tilde{u}\}
    \|_{N(I)}\leq \epsilon<\epsilon_1,\\
    &\|e^{i(t-t_0)\Delta}[u_0-\tilde{u}(t_0)]\|_{S^1(I)}\leq \epsilon<\epsilon_1,
  \end{align*}
  for some $t_0\in I$ and $u_0\in\dot{H}^1(\R^3)$, then there exists a unique solution $u:I\times\R^3\rightarrow\mathbb{C}$ with $u(t_0)=u_0$, which satisfies
  $$\|u-\tilde{u}\|_{S^1(I)}\lesssim \epsilon.$$
\end{proposition}
\begin{proof}
Without loss of generality, we may assume that $I=[0,a)$. For any $\eta>0$ small enough, there exists $I_j=[a_j,a_{j+1})$ such that $I=\bigcup_{j=1}^{J}I_j$ ($J<\infty$) and $\|u\|_{S^1(I)}\leq \eta$. First note by using the Strichartz estimate, Lemma \ref{Non-e1} and a standard continuity argument we have $\|\nabla\widetilde{u}\|_{W(I)}\lesssim M.$

  Let us  define $u=\tilde{u}+w$, so that the equation for $w$ is written as
  \begin{align*}
    \begin{cases}
      iw_t+\Delta w=F(\tilde{u}+w)-F(\tilde{u})+e,\\
      w(0)=u_0-\tilde{u}(0),
    \end{cases}
  \end{align*}
  where $F(u)=-(I_\alpha\ast|\cdot|^{-b}|u|^p)|x|^{-b}|u|^{p-2}u$ and $e=(i\partial_t+\Delta)\tilde{u}+|x|^{-b}(I_{\alpha}\ast|\cdot|^{-b}|\tilde{u}|^p)|\tilde{u}|^{p-2}\tilde{u}$. Thus, using Duhamel on $I_j$, $w$ satisfies
  $$w(t)=e^{i(t-a_j)\Delta}w(a_j)-i\int_{a_j}^te^{i(t-s)\Delta}(F(\tilde{u}+w)-F(\tilde{u}))ds-
  i\int_{a_j}^te^{i(t-s)\Delta}eds.$$
  The proof of Proposition \ref{LWP} shows that shows that, for $\eta$ and $\epsilon_1$ small enough, one can obtain that there exits a unique solution $w(t)$ on $I_1$ such that $\|w\|_{S^1(I_1)}\leq 2\epsilon$ and $\|\nabla w\|_{W(I_1)}\leq C_1(E,\|u_0\|_{\dot H^1}).$ Together with Strichartz estimate, one can obtain $\|w\|_{L_{I_1}\dot H^1}\leq C_1'(E,\|u_0\|).$

  Now we turn to estimate $\|e^{i(t-a_2)\Delta}w(a_2)\|_{S^1(I_2)}$. In fact,
  \begin{align*}
    e^{i(t-a_2)\Delta}w(a_2)=e^{it(a_2-a_0)\Delta}-i\int_{a_1}^{a_2}e^{i(t-s)\Delta}(F(\tilde{u}+w)-F(\tilde{u}))ds-
  i\int_{a_1}^{a_2}e^{i(t-s)\Delta}eds.
  \end{align*}
  Thus using the Strichartz estimate, $\|w\|_{S^1(I_1)}\leq 2\epsilon$ and $\|\nabla w\|_{W(I_1)}\leq C(E,\|u_0\|_{\dot H^1})$, we estimate
  $$\|e^{i(t-a_2)\Delta}w(a_2)\|_{S^1(I_2)}\leq 2\epsilon .$$
  Reiterating above argument, one has
  $$\|w\|_{S^1(I_2)}\leq 4\epsilon\;\;\text{and}\;\; \|\nabla w\|_{W(I_2)}\leq C_2(E,\|u_0\|_{\dot H^1}).$$

 By an inductive argument, then we have, for $j=1,...,J$,
 $$\|w\|_{S^1(I_j)}\leq 2^j\epsilon\;\; \text{and}\;\; \|\nabla w\|_{W(I_2)}\leq C_j(E,\|u_0\|_{\dot H^1}).$$
 Therefore, summing them over all subintervals $I_j$, we get the desired.
\end{proof}

\section{Profile decomposition and variational analysis}

This section contains a linear profile decomposition result. It will be very important to show the new ingredient of this paper (Proposition \ref{N-profile}) and the existence of a critical solution. Furthermore, we study some properties that are related to our problem. First, the Gagliardo--Nirenberg inequality and in the sequel the energy trapping.

\begin{proposition}[Linear profile decomposition]\label{Linear-Profile}
  Let $u_n$ be a bounded sequence in $\dot{H}^1(\R^3)$. Then the following holds up to a subsequence:

  There exist $J^*\in \mathbb{N}\cup\{\infty\}$; profiles $\phi^j\in \dot{H}^1\setminus\{0\}$; scales $\lambda_n^j\in(0,\infty)$; space translation parameters $x_n^j\in\R^3$; time translation parameters $t_n^j$; and remainders $w_n^J$ so that writing
  $$g_n^jf(x)=(\lambda_n^j)^{-\frac12}f(\frac{x-x_n^j}{\lambda_n^j}),$$
  we have the following decomposition for $1\leq j\leq J^*$:
  $$u_n=\sum_{j=1}^{J}g_n^j[e^{it_n^j\Delta}\phi^j]+W_n^J.$$
  This decomposition satisfies the following conditions:
  \begin{itemize}
    \item Energy decoupling: writing $P(u)=\||x|^{-b}(I_\alpha\ast|\cdot|^{-b}|u|^p)|u|^p\|_{L^1}$, one has
    \begin{align}\label{energy-decoupling}
      &\lim_{n\to\infty}\{\|\nabla u_n\|_{L^2}^2-\sum_{j=1}^J\|\nabla \phi^j\|_{L^2}^2-\|\nabla W_n^J\|_{L^2}^2\}=0\\\label{energy-decoupling-1}
      &\lim_{n\to\infty}\{P(u_n)-\sum_{j=1}^JP(g_n^j[e^{it_n^j\Delta}\phi^j])-P(W_n^J)\}=0.
    \end{align}
    \item Asymptotic vanishing of remainders:
    \begin{align}\label{vanishing-condition}
      \limsup_{J\to J^*}\limsup_{n\to\infty}\|e^{it\Delta}W_n^J\|_{S^1(\R)}=0.
    \end{align}
    \item Asymptotic orthogonality of parameters: for $j\neq k$,
    \begin{align}\label{ZJ}
      \lim_{n\to\infty}\left\{\log \left[\frac{\lambda_n^j}{\lambda_n^k}\right]+\frac{|x_n^j-x_n^k|^2}{\lambda_n^j\lambda_n^k}
      +\frac{|t_n^j(\lambda_n^j)^2-t_n^k(\lambda_n^k)^2|}{\lambda_n^j\lambda_n^k}\right\}=\infty.
    \end{align}
  \end{itemize}
  In addition, we may assume that either $t_n^j\equiv0$ or $t_n^j\to\pm\infty$, and that either $x_n^j\equiv0$ or $|x_n^j|\to\pm\infty$.
\end{proposition}
\begin{proof}
  We only show the potential energy decoupling \eqref{energy-decoupling-1}, the other cases can be shown by the same way as in \cite{Ker}. By induction, it suffices to prove
  \begin{align}\label{Potential-decoupling}
    \lim_{n\to\infty}\left(P(u_n)-P(u_n-\phi_n^1)-P(\phi_n^1)\right)=0,
  \end{align}
where $\phi_n^1=g_n^1(e^{it_n^1\Delta}\phi^j).$

Suppose that $|t_n^1|\to\infty$. Then using the dispersive estimate such that $\|e^{it_n^1\Delta}f\|_{L^p}\to0$ for any $p>2$, we obtain from \eqref{Potential-decoupling} that
\begin{align*}
  &P(u_n)-P(u_n-\phi_n^1)-P(\phi_n^1)\\
  \leq&\int_{\R^3}[I_\alpha\ast|\cdot|^{-b}(|u_n|^p-|u_n-\phi_n^1|^p)]|x|^{-b}|u|^pdx\\
  &+\int_{\R^3}(I_\alpha\ast|\cdot|^{-b}||u_n-\phi_n^1|^p)|x|^{-b}(|u_n|^p-|u_n-\phi_n^1|^p)dx\\
  &+\int_{\R^3}(I_{\alpha}\ast|\cdot|^{-b}|\phi_n^1|^p)|x|^{-b}|\phi_n^1|^p dx\\
  \lesssim&\||x|^{-1}u_n\|_{L^2}^{2b}\|u_{n}\|_{L_x^6}^{2p-1-2b}\|\phi_{n}^1\|_{L^6}+
  \||x|^{-1}\phi_n^1\|_{L_x^2}^{2b}\|\phi_{n}^1\|_{L^6}^{2p-2b}\\
  \to&\,0\ \textnormal{as}\ n\to\infty.
\end{align*}

We now assume that $t_n^1\equiv0$. Then we get that $(g_n^1)^{-1}\phi_n^1=\phi^1$ and
\begin{align}
  (g_n^1)^{-1}u_n\to\phi^1\ weakly\ in\ \dot{H}^1\ as\ n\to\infty.
\end{align}
By scaling, we deduce that
\begin{align*}
&P(u_n)-P(u_n-\phi_n^1)-P(\phi_n^1)\\
=&\int_{\R^3}(I_\alpha\ast|\cdot-x_n^1/\lambda_n^1|^{-b}|(g_n^1)^{-1}u_n|^p)
|x-x_n^1/\lambda_n^1|^{-b}|(g_n^1)^{-1}u_n|^pdx\\
&-\int_{\R^3}(I_{\alpha}\ast|\cdot-x_n^1/\lambda_n^1|^{-b}|(g_n^1)^{-1}u_n-\phi^1|^p)
|x-x_n^1/\lambda_n^1|^{-b}|(g_n^1)^{-1}u_n-\phi^1|^pdx\\
&-\int_{\R^3}(I_{\alpha}\ast|\cdot-x_n^1/\lambda_n^1|^{-b}|\phi^1|^p)|x-x_n^1/\lambda_n^1|^{-b}|\phi^1|^p dx\\
=&A_n+B_n+C_n+D_n,
\end{align*}
where
\begin{align*}
  A_n=&\int_{\R^3}[I_\alpha\ast|\cdot-x_n^1/\lambda_n^1|^{-b}(|(g_n^1)^{-1}u_n|^p-|(g_n^1)^{-1}u_n-\phi^1|^p)]
|x-x_n^1/\lambda_n^1|^{-b}|(g_n^1)^{-1}u_n|^pdx\\
&-\int_{\R^3}(I_{\alpha}\ast|\cdot-x_n^1/\lambda_n^1|^{-b}|\phi^1|^p)|x-x_n^1/\lambda_n^1|^{-b}|(g_n^1)^{-1}u_n|^pdx.
\end{align*}
\begin{align*}
B_n=&\int_{\R^3}(I_{\alpha}\ast|\cdot-x_n^1/\lambda_n^1|^{-b}|(g_n^1)^{-1}u_n-\phi^1|^p)
|x-x_n^1/\lambda_n^1|^{-b}(|g_n^{-1}u_n|^p-|(g_n^1)^{-1}u_n-\phi^1|^p)dx\\
&-\int_{\R^3}(I_{\alpha}\ast|\cdot-x_n^1/\lambda_n^1|^{-b}|(g_n^1)^{-1}u_n-\phi^1|^p)|x-x_n^1/\lambda_n^1|^{-b}|\phi^1|^p.
\end{align*}
\begin{align*}
  C_n=&-\int_{\R^3}[I_\alpha\ast|\cdot-x_n^1/\lambda_n^1|^{-b}(|(g_n^1)^{-1}u_n-\phi^1|^p+|\phi^1|^p)]
|x-x_n^1/\lambda_n^1|^{-b}|\phi^1|^pdx\\
&+\int_{\R^3}(I_{\alpha}\ast|\cdot-x_n^1/\lambda_n^1|^{-b}|(g_n^1)^{-1}u_n|^p)|x-x_n^1\lambda_n^1|^{-b}|\phi^1|^p dx.
\end{align*}
\begin{align*}
  D_n=2\int_{\R^3}(I_{\alpha}\ast|\cdot-x_n^1/\lambda_n^1|^{-b}|\phi^1|^p)|x-x_n^1/\lambda_n^1|^{-b}|(g_n^1)^{-1}u_n-\phi^1|^pdx.
\end{align*}
 Applying the density by $C_c^\infty(\R^3)$, we assume that $\phi^1\in C_c^\infty(\R^3)$. Following the argument in \cite{CHL}, we can get
 $$A_n,B_n,C_n\to0\ \ as\ \ n\to\infty.$$
 It suffices to show $D_n\to\infty$. If $\frac{|x_n^1|}{\lambda_n^1}\to\infty$, let $\chi_n(x)$ be smoothing function denoted by
   \begin{equation*}
    \chi_n(x)=
    \begin{cases}
      0,\ \ |x-\frac{x_n^1}{\lambda_n^1}|\geq\frac12|\frac{x_n^1}{\lambda_n}|\\
      1,\ \ |x-\frac{x_n^1}{\lambda_n^1}|<\frac14|\frac{x_n^1}{\lambda_n}|
    \end{cases},
 \end{equation*}
 thus
 $$\|\nabla(\chi_n\phi^1)\|_{L^2}\to0\ \ as\ \ n\to0.$$
 Then
 \begin{align*}
   D_n\lesssim & \||x-x_n^1/\lambda_n^1|^{-1}(1-\chi_n)\phi^1\|_{L_x^2}^b+\||x-x_n^1/\lambda_n^1|^{-1}\chi_n\phi^1\|_{L_x^2}^b\\
   \lesssim& \frac{|x_n^1|}{\lambda_n^1}+\||x-x_n^1/\lambda_n^1|^{-1}\chi_n\phi^1\|_{L_x^2}^b\to0\ \ as\ \ n\to\infty.
 \end{align*}
If $ \frac{x_n^1}{\lambda_n^1}\to x_0$, we assume that $\frac{x_n^1}{\lambda_n^1}\equiv0$. Since
$$\|I_\alpha\ast|\cdot|^{-b}|\phi^1|^p\|_{L^{\frac{6}{3-\alpha}}}\lesssim \|\phi^1\|_{\dot{H}^1}^p,$$
we can choose $\beta\in C_c^\infty(\R^3\setminus\{0\})$ close to $I_\alpha\ast|\cdot|^{-b}|\phi^1|^p$ in $L^{\frac{6}{3-\alpha}}$. By the density, it suffices to show
$$\int_{\R^3}\beta(x)|x|^{-b}|(g_n^1)^{-1}u_n-\phi^1|^p(x)dx\to0\ \ as\ \ n\to\infty.$$
Indeed, H\"older's inequality leads to
\begin{align*}
\int_{\R^3}\beta(x)|x|^{-b}|(g_n^1)^{-1}u_n-\phi^1|^p(x)dx\lesssim & \int_{\R^3}\beta(x)|(g_n^1)^{-1}u_n-\phi^1|^p(x)dx\\
\lesssim& \|(g_n^1)^{-1}u_n-\phi^1\|_{L^6}^{p-1}\|\beta((g_n^1)^{-1}u_n-\phi^1)\|_{L^{\frac{6}{7-p}}}\\
\lesssim&\|\beta((g_n^1)^{-1}u_n-\phi^1)\|_{L^{\frac{6}{7-p}}}\to0 \ \ as\ \ n\to\infty,
\end{align*}
where we use the compactness of multiplication operator by $\beta$ from $\dot{H}^1$ to $L^r$ with $1\leq r<6$. So, we complete the proof.
\end{proof}
\begin{lemma}\label{profiles}
  Let $f_n$ be a bounded sequence in $\dot{H}^1(\R^3)$. Then there exist $J^*\in \mathbb{N}\cup\{\infty\}$; profiles $\phi^j\in \dot{H}^1\setminus\{0\}$; scales $\lambda_n^j\in(0,\infty)$; space translation parameters $x_n^j\in\R^3$ such that writing
  $$f_n(x)=\sum_{j=1}^J(\lambda_n^j)^{-\frac12}\phi^j\Big(\frac{x-x_n^j}{\lambda_n^j}\Big)+r_n^J(x), \ \ 1\leq J\leq J^*.$$
  The decomposition has the following properties:
  $$\limsup_{J\to J^*}\limsup_{n\to\infty}\|r_n^J\|_{L_x^6}=0,$$
  $$\limsup_{n\to\infty}\Big|\|f_n\|_{\dot H^1}^2-\sum_{j=1}^J\|\phi^j\|_{\dot H^1}^2-\|r_n^J\|_{\dot H^1}^2\Big|=0,$$
  $$\liminf_{n\to\infty}\Big[\frac{|x_n^j-x_n^k|^2}{\lambda_n^j\lambda_n^k}
  +\log\frac{\lambda_n^k}{\lambda_n^j}\Big]=\infty,\ \forall j\neq k,$$
  $$(\lambda_n^j)^{\frac12}r_n^J(\lambda_n^j+x_n^j)\to0,\ weakly\ in\ \dot H^1,$$
  \begin{align}\label{jianjin-P}
   \limsup_{n\to\infty}\liminf_{n\to\infty}\Big[P(f_n)-\sum_{j=1}^JP(g_n^j\phi^j)-P(r_n^J)\Big]=0.
  \end{align}
\end{lemma}
\begin{proof}
  Following the argument in \cite{Ja}, it suffices to show \eqref{jianjin-P}. In fact, the proof of \eqref{jianjin-P} comes from \eqref{energy-decoupling-1} in Proposition \ref{Linear-Profile}.
\end{proof}

Now, we show the following sharp Gagliardo--Nirembergy inequality.
\begin{proposition}[Sharp Gagliardo--Nirenberg inequality]
  Let $p\geq \max\{b,2\}$. There exists a positive constant $C_0$ such that for any $u\in\dot{H}^1(\R^3)$,
  \begin{align}\label{GN-inequality}
    \int_{\R^3}(I_\alpha\ast|\cdot|^{-b}|u|^p)|x|^{-b}|u|^pdx\leq C_0\|u\|_{\dot{H}^1}^{2p}.
  \end{align}
  Moreover, there exists a radial non-negative function $W\in\dot{H}^1(\R^3)$ so that
  $$\int_{\R^3}(I_\alpha\ast|\cdot|^{-b}|W|^p)|x|^{-b}|W|^pdx=C_0\|W\|_{\dot{H}^1}^{2p},$$
  and $W$(ground state) is the solution of the elliptic equation:
  \begin{align}\label{elliptic}
 \Delta W+(I_\alpha\ast|\cdot|^{-b}|W|^p)|x|^{-b}|W|^{p-2}W=0.
  \end{align}
\end{proposition}
\begin{proof}
First, we prove the Gagliardo--Nirenberg inequality \eqref{GN-inequality}. Lemma \ref{HLS} implies that
  \begin{align*}
   \int_{\R^3}(I_\alpha\ast|\cdot|^{-b}|u|^p)|x|^{-b}|u|^pdx\leq& C\||x|^{-b}|u|^p\|_{L^{\frac{6}{3+\alpha}}}^2.
  \end{align*}
  Combining the relation
  $$\frac{3+\alpha}{6}=\frac{b}2+\frac{p-b}{6},$$
  together with H\"older's and  Hardy's inequalities, it follows that
 \begin{align*}
   \||x|^{-b}|u|^p\|_{L^{\frac{6}{3+\alpha}}}\leq& \||x|^{-1}u\|_{L^2}^b\|u\|_{L^6}^{p-b}\\
   \leq&C\|u\|_{\dot{H}^1}^{p},
  \end{align*}
  thus, we complete \eqref{GN-inequality}.

 Next, we turn to prove the sharpness for the inequality \eqref{GN-inequality}. For any $f\in \dot H^1(\R^3)$, let $P(f):=\int_{\R^3}(I_\alpha\ast|\cdot|^{-b}|f|^p)|x|^{-b}|f|^pdx$. Define the functional $J(f)$ in $\dot H^1(\R^3)$:
  \begin{align}
    J(f)=\frac{P(f)}{\|\nabla f\|_{L^2}^{2p}}.
  \end{align}
 By the invariant of $J(f)$ under the scaling $f_\lambda(x)=:f(\lambda x)$, there exists a sequence $\{f_n\}$ satisfying $\|f_n\|_{\dot H^1}=1$ and
  \begin{align*}
    \lim_{n\to\infty}J(f_n)=\sup_{f\in \dot{H}^1\setminus\{0\}}J(f)=C_0.
  \end{align*}
  Using the Schwartz symmetrical rearrangement lemma, without loss of generality, we also assume $f_n$ are non-negative radial.

Due to Lemma \ref{profiles}, we obtain
  $$f_n=\sum_{j=1}^{M}g_n^j[\phi^j]+r_n^M.$$
 We claim that there exists one profile, that is  $\phi^j\equiv0$ for $j\geq2$.
In fact, by the asymptotic vanishing of potential in Lemma \ref{profiles}, one has
 \begin{align*}
   C_0\leq \lim_{n\to\infty}\sum_{j=1}^{\infty}P(g_n^j[\phi^j])\leq C_0\sum_{j=1}^{\infty}\|\nabla \phi^j\|_{L^2}^{2p}.
 \end{align*}
On the other hand, we get the fact
 $$\sum_{j=1}^{\infty}\|\nabla \phi^j\|_{L^2}^{2}\leq 1.$$
 If there are more than one profile $\phi^j\neq0$, since $p>1$ we have
 $$\sum_{j=1}^{\infty}\|\nabla \phi^j\|_{L^2}^{2p}<1, $$
 thus $C_0<C_0$, this is contradiction. Hence, there exists only one profile and $\|\nabla \phi^1\|_{L^2}=1$.

Since
 $(g_n^1)^{-1}[f_n](x)\rightharpoonup \phi^1$ weakly in
$\dot H^1$ and $\|f_{n}\|_{\dot H^1}=\|\phi^1\|_{\dot H^1}$, we deduce that
 $$\lim_{n\to\infty}\|(g_n^1)^{-1}[f_n]-\phi^1\|_{\dot H^1}=0.$$
 Hence, $C_0=J(\phi^1)$.

Finally, let $\lambda=\|\phi^1\|_{\dot H^1}^{\frac{1}{2}}$, we define
 $$W(x)=\phi^1(\lambda^{-1}x).$$
  We can prove $W$ solve the elliptic equation \eqref{elliptic} via the variational derivatives of $J(f)$ (See \cite{Ta,We} for more details, we omit the proof). Then we complete the proof of this statement.
 \end{proof}

The following lemma implies that the solutions to \eqref{INLC} obey the sub-threshold condition \eqref{threshold} throughout their lifespan.
\begin{lemma}[Energy trapping,\cite{CHL,GM}]\label{energy-tra}
  Suppose that $u:I\times\R^3\to\mathbb{C}$ is a solution to \eqref{INLC} with initial data $u_0$ obeying \eqref{threshold}.
  Then there exists $\delta>0$ such that
  $$\sup_{t\in I}\|u(t)\|_{\dot{H}^1}<(1-\delta)\|W\|_{\dot{H}}.$$
  Furthermore,
  \begin{align}
    E(u(t))\sim\|u(t)\|_{\dot{H}^1}^2\sim\|u_0\|_{\dot{H}^1}^2\ \ for\ \ all\ \ t\in I.
  \end{align}
\end{lemma}

\section{Existence of critical solution}
In this section, we show that the existence of minimal energy blow-up solution following the argument of Kenig--Merle \cite{KM} if Theorem \ref{theorem} fails. The difficulty here is how to construct the nonlinear profiles.

We first give a result that is the key result of this paper. It establishes the existence of scattering solutions to \eqref{INLC} associated with initial data living sufficiently far from the origin. This new ingredient allows us to consider the non-radial setting. It is worth mentioning, for the homogeneous model, this new ingredient is not applied because the nonlinearity doesn't have decay. Thus, we assume radial data for studying scattering for the case ($b=0$).

\begin{proposition}\label{N-profile}
  Let $\lambda_n\in(0,\infty), x_n\in\R^3$, and $t_n\in \R$ satisfy
  $$\lim_{n\to\infty}\frac{|x_n|}{\lambda_n}=\infty\ \ and\ \ t_n\equiv0\ \ or\ \ t_n\to\pm\infty.$$
  Let $\phi\in\dot{H}^1(\R^3)$ and define
  $$\phi_n(x)=g_n[e^{it_n\Delta}\phi](x)=\lambda_n^{-\frac12}[e^{it_n\Delta}\phi](\frac{x-x_n}{\lambda_n}).$$
  Then for all $n$ sufficiently large, there exists a global solution $v_n$ to \eqref{INLC} satisfying
  $$v_n(0)=\phi_n,\ \|v_n\|_{S(\R)}\lesssim 1,$$
  with implicit constant depending only on $\|\phi\|_{\dot{H}^1}$.

  Furthermore, for any $\epsilon>0$ there exists $N\in\mathbb{N}$ and $\psi\in C_c^\infty(\R\times\R^3)$ so that for $n\geq N$, we have
  \begin{align}\label{smoothing}
    \|v_n-T_n[\psi]\|_{S(\R)}<\epsilon,
  \end{align}
  where
  $$T_nf(t,x)=\lambda_n^{-\frac12}f\big(\frac{t}{\lambda_n^2}+t_n,\frac{x-x_n}{\lambda_n}\big).$$
\end{proposition}

\begin{proof}
  Let $\theta\in(0,1)$ be a small parameter to be determined below and introduce a frequency cutoff $P_n$ and spatial cutoff $\chi_n$ as follows. First, we let
  $$P_n=P_{\left|\frac{x_n}{\lambda_n}\right|^{-\theta}\leq\cdot\leq\left|\frac{x_n}{\lambda_n}\right|^{\theta}}.$$
  Next, we take $\chi_n$ to be a smooth function satisfying
  \begin{equation*}
    \chi_n(x)=
    \begin{cases}
      1,\ \ |x+\frac{x_n}{\lambda_n}|\geq\frac12|\frac{x_n}{\lambda_n}|\\
      0,\ \ |x+\frac{x_n}{\lambda_n}|<\frac14|\frac{x_n}{\lambda_n}|,
    \end{cases}
  \end{equation*}
  with $|\partial^{\alpha}\chi_n|\leq C|\frac{x_n}{\lambda_n}|^{-|\alpha|}$ for all multi-indices $\alpha$ and $\chi_n\to1$ as $n\to\infty$.

  Now, we define the approximations $v_n,T$. Let
  $$I_{n,T}:=[a_{n,T}^-,a_{n,T}^+]:=[-\lambda_n^2t_n-\lambda_n^2T,-\lambda_n^2t_n+\lambda_n^2T]$$
  and for $t\in I_{n,T}$ define
  \begin{align*}
    v_{n,T}(t)&=g_n[\chi_nP_ne^{i(\lambda_n^{-2}t+t_n)\Delta}\phi]\\
    &=\chi_n(\frac{x-x_n}{\lambda_n})e^{i(t+\lambda_n^2t_n)\Delta}g_n[P_n\phi].
  \end{align*}

  Next, let
  $$I_{n,T}^+:=(a_{n,T}^+,\infty),\ \ I_{n,T}^-:=(-\infty,a_{n,T}^-)$$
  and set
  \begin{equation*}
    v_{n,T}(t)=
    \begin{cases}
      e^{i(t-a_{n,T}^+)\Delta}[v_{n,T}(a_{n,T}^+)]\ \ t\in I_{n,T}^+,\\
      e^{i(t-a_{n,T}^-)\Delta}[v_{n,T}(a_{n,T}^-)]\ \ t\in I_{n,T}^-.
    \end{cases}
  \end{equation*}

  \textbf{Condition 1:}
  $$\limsup_{T\to\infty}\limsup_{n\to\infty}\{\|v_{n,T}\|_{L_t^\infty\dot{H}^1(\R\times\R^3)}+
  \|v_{n,T}\|_{S(\R)}\}\lesssim1.$$
  We estimate separately on $I_{n,T}$ and $I_{n,T}^{\pm}$. We first estimate on $I_{n,T}$. Noting that
  $$\|\chi_{n}\|_{L^\infty}+\|\nabla\chi_n\|_{L^3}\lesssim1,$$
  then using Strichartz and Sobolev embedding we have
  \begin{align*}
\|v_{n,T}\|_{L_t^\infty\dot{H}^1(\R\times\R^3)}\leq C\|\nabla(P_ne^{i(\lambda_n^{-2}+t_n)\Delta}\phi)\|_{L^2}+\|P_ne^{i(\lambda_n^{-2}+t_n)\Delta}\phi\|_{L^6}\lesssim1
  \end{align*}
and
$$\|v_{n,T}\|_{S(\R)}\lesssim1.$$
With the $L_t^\infty\dot{H}^1$ bound in place on $I_{n,T}$, the desired bounds $I_{n,T}^{\pm}$ follow from Sobolev and Strichartz estimates.

\textbf{Condition 2:}
$$\lim_{T\to\infty}\limsup_{n\to\infty}\|v_{n,T}(0)-\phi_n\|_{\dot{H}^1}=0.$$
We treat the following two cases:\\
 (i) $t_{n}\equiv0$ (then $0\in I_{n,T}$),\\
 (ii) $t_{n}\to\pm\infty$ (so that $0\in I_{n,T}^\pm$ for $n$ sufficiently large).

 Case (i). If $t_{n}\equiv0$, then
 $$\|v_{n,T}(0)-\phi_n\|_{\dot{H}^1}=\|(\chi_nP_n-1)\phi\|_{\dot{H}^1}\to0,\ n\to\infty.$$

 Case (ii). If $t_n\to\infty$, we may write
 $$v_{n,T}(0)=g_ne^{it_n\Delta}[e^{-itT\Delta}\chi_nP_ne^{iT\Delta}\phi].$$
 Then we obtain
 \begin{align*}
   \|v_{n,T}(0)-\phi_n\|_{\dot{H}^1}&=\left\|[e^{-itT\Delta}\chi_nP_ne^{iT\Delta}-1]\phi\right\|_{\dot{H}^1}\\
   &=\|[\chi_nP_n-1]e^{iT\Delta}\phi\|_{\dot{H}^1}\to0
 \end{align*}
 as $n\to\infty$.

\textbf{Condition 3:} Denoting
$$e_{n,T}=(i\partial_t+\Delta)v_{n,T}+|x|^{-b}(I_{\alpha}\ast|\cdot|^{-b}|v_{n,T}|^p)|v_{n,T}|^{p-2}v_{n,T}.$$
Note that
$$\lim_{T\to\infty}\limsup_{n\to\infty}\|\nabla e_{n,T}\|_{N(\R)}=0.$$

We first consider the interval $I_{n,T}$. Defining $e_{n,T}=e^{lin}_{n,T}+e_{n,T}^{nl}$, with
\begin{align*}
e^{lin}_{n,T}=&\Delta[\chi_n(\frac{x-x_n}{\lambda_n})]e^{i(t+\lambda_n^2t_n)\Delta}g_n[P_n\phi]\\
&+2\nabla[\chi_n(\frac{x-x_n}{\lambda_n})]e^{i(t+\lambda_n^2t_n)\Delta}\nabla g_n[P_n\phi]
\end{align*}
and
$$e_{n,T}^{nl}=\lambda_n^{-(p-1)+\alpha}g_n\left\{|\lambda_nx+x_n|^{-b}\chi_n^{p-1}
(I_{\alpha}\ast|\lambda_n\cdot+x_n|^{-b}|\chi_n\Phi_n|^p)
|\Phi_n|^{p-2}\Phi_n\right\}$$
where $\Phi_n(t,x)=P_ne^{i(\lambda_n^{-2}t+t_n)\Delta}\phi$.

Observe that the gradient to $e_{n,T}^{lin},$ is a sum of terms of the form
$$\partial^j[\chi_{n}(\frac{x-x_n}{\lambda_n})]e^{i(t+\lambda_n^2t_n)\Delta}\partial^{3-j}[g_nP_n\phi]\ \ \textnormal{for}\ \ j\in\{1,2,3\}.$$
By H\"older's inequality, Bernstein's inequality, it follows that

\

$
\quad \quad\left\| \partial^j[\chi_{n}(\frac{x-x_n}{\lambda_n})]e^{i(t+\lambda_n^2t_n)\Delta}\partial^{3-j}[g_nP_n\phi]\right\|_{L_t^1L_x^2(I_{n,T}\times\R^3)}
$
\begin{align*}
  &\lesssim |I_{n,T}|\left\|\partial^j[\chi_{n}(\frac{x-x_n}{\lambda_n})]\right\|_{L_x^\infty}\left\|\partial^{3-j}[g_nP_n\phi]\right\|_{L_t^\infty L_x^2(I_{n,T}\times\R^3)}\\
  &\lesssim T\left|\frac{x_n}{\lambda_n}\right|^{-j}\left\|\partial^{3-j}[g_nP_n\phi]\right\|_{L_t^\infty L_x^2(I_{n,T}\times\R^3)}\\
  &\lesssim T|\frac{x_n}{\lambda_n}|^{-j}|\frac{x_n}{\lambda_n}|^{|2-j|\theta}\to 0,
\end{align*}
as $n\to\infty$ for $\theta$ sufficiently small.

To estimate $e_{n,T}^{nl}$ on $I_{n,T}$. We choose $r$ such that
$$1+\frac{\alpha}{3}=\frac16+\frac{2(p-1)}r+\frac{1}{2}.$$
H\"older's inequality leads to
$$\|\nabla e_{n,T}^{nl}\|_{L_t^{2}L_x^{\frac65}}\leq \lambda_{n}^{2b}T^{1/2}\left\|\nabla\left[|\lambda_nx+x_n|^{-b}\chi_{n}^{p-1}(I_{\alpha}\ast|\lambda_n\cdot+x_n|^{-b}|\chi_n\Phi_n|^p)|\Phi_n|^{p-2}\Phi_n\right]
\right\|_{L_t^{\infty}L_x^{\frac65}}.$$
Noting that
$$\left\|\partial^j[|\lambda_nx+x_n|^{-b}\chi_n^{p-1}]\right\|_{L_x^\infty}\lesssim |\frac{x_n}{\lambda_n}|^{-j}|x_n|^{-b},\ \ j\in\{0,1\}.$$
Hence, applying Hardy--littlewood--Sobolev's inequality, H\"older's inequality, Sobolev embedding, and Bernstein, we obtain
\begin{align*}
  \|\nabla e_{n,T}^{nl}\|_{L_t^{2}L_x^{\frac65}}\lesssim& T^{1/2}\big|\frac{x_n}{\lambda_n}\big|^{-2b}\|\nabla \Phi_n\|_{L_t^\infty L_x^r}^{2(p-1)}\sum_{j=0}^{1}\big|\frac{x_n}{\lambda_n}\big|^{-j}\|\partial^{1-j}\Phi_n\|_{L_t^\infty L_x^r}\\
  \lesssim& T^{1/2}\big|\frac{x_n}{\lambda_n}\big|^{-2b+\theta|5/2-3/r|}\to0
\end{align*}
as $n\to\infty$ for $\theta$ sufficiently small.
The estimate on the interval $I_{n,T}^\pm$ is easy by Sobolev embedding and Strichartz, we omit the proof.

Proposition \ref{Stability} implies that there exist solutions $v_n$ with initial data $\phi_n$, and these solutions have the following estimates
\begin{align}\label{Sca-bound}
  \|v_n\|_{S(\R)}\leq C\|\phi\|_{\dot{H}^1}\;\; \textnormal{and}\;\; \limsup_{T\to\infty}\lim_{n\to\infty}\|v_n-v_{n,T}\|_{S(\R)}=0.
\end{align}

With \eqref{Sca-bound} in place, we can adapt the same arguments from \cite{KMVZZ} to prove the approximation by $C_c^\infty$ functions. Here, we omit the proof.
\end{proof}

\begin{remark}
The previous result as was mentioned above will allow us to prove scattering of \eqref{INLC} assuming non-radial initial data. The decay in the nonlinearity is essential in the proof. For the homogeneous model, it is necessary to assume radial data to establish it.
\end{remark}

Next, using the proposition \ref{N-profile} and the stability result, we can show that there exists minimal energy blow-up solution to \eqref{INLC} if Theorem \ref{theorem} fails.
\begin{definition}
  For each $0<c<E(W)$, define
  $$L(c):=\{\phi\in \dot{H}^1: E(\phi)<c, \|\phi\|_{\dot{H}^1}<\|W\|_{\dot{H}^1}\}.$$
\end{definition}
We also define the $critical$-$index$ $E_c$, denoted by
\begin{align}
 E_c=\sup\{c: v_0\in L(c), \|v\|_{S(I_{\max})}<\infty\},
\end{align}
where $v$ to solution \eqref{INLC} and  $I_{\max}$ denote the maximal lifespan.
\begin{remark}
  We claim that $E_c>0$. In fact, if $E(v_0)<\delta$($\delta$ is small enough), by the energy-trapping, we can get
$$\|v_0\|_{\dot{H}^1}\lesssim \delta^{\frac{1}{2}}.$$
Using Proposition \ref{LWP}, we obtain that  $v(t)$ is global and scatters, and $\|v\|_{S(\R)}<\infty.$
\end{remark}
\begin{remark}
  In fact, the theorem \ref{theorem} is equivalent to $E_c=E(W)$. We will show this fact by a contradiction.
  \end{remark}

\begin{proposition}\label{MEBS}
  Assume Theorem \ref{theorem} fails, then there exists $\phi_c\in \dot{H}^1$ such that $E(\phi_c)=E_c$ and $$\|\phi_c\|_{\dot{H}^1}<\|W\|_{\dot{H}^1}.$$
  Let $u_c$ be the solution to \eqref{INLC} with initial data $\phi_c$, then
  $$\|u_c\|_{S(I_{\max})}=\infty.$$
\end{proposition}
\begin{definition}
  Let $u_0\in \dot{H}^1$ and $\{t_n\}$ be a sequence with $t_n\to \bar{t}\in [-\infty,\infty]$. We say that $u(t)$ is a nonlinear profile associated with $(u_0,\{t_n\})$ if there exist a maximal interval $I_{\max}$ with $\bar{t}\in I_{\max}$ such that $u$ is solution of \eqref{INLC} on $I_{\max}$ and
  $$\|u(t_n)-e^{it_n\Delta}u_0\|_{\dot{H}^1}\to0$$
  as $n\to\infty$.
\end{definition}
\begin{remark}
  We claim the nonlinear profile always exists. In fact, we choose $u$ be the solution to \eqref{INLC} with initial data $e^{i\bar{t}\Delta}u_0$ when $|\bar{t}|<\infty.$ If $\bar{t}=\pm\infty$, define $u(t)$ by
  $$u(t)=e^{it\Delta}u_0-i\int_{t}^{\pm \infty}e^{i(t-s)\Delta}(|x|^{-b}(I_\alpha\ast|\cdot|^{-b}|u|^p)|u|^{p-2}u)(s)ds$$
  which solves the equation \eqref{INLC}. Using Strichartz estimates, we can get $u$ scatters forward/backward in time, i.e.,
  $$\|u(t)-e^{it\Delta}u_0\|_{\dot{H}^1}\to 0, \ as\ t\to\infty.$$
\end{remark}

Next, we turn to prove the proposition \ref{MEBS} via nonlinear profile and stability.
\begin{proof}[Proof of Proposition \ref{MEBS}]
By the definition of $E_c$, there exists a sequence $\{u_n\}$ such that $\|u_n(0)\|_{\dot{H}^1}<\|W\|_{\dot{H}^1}$, $\|u_n\|_{S(I_{\max}^n)}\to\infty$ and
$$E(u_n)\to E_c$$
as $n\to\infty$. Applying Proposition \ref{Linear-Profile}, we get
\begin{align}
  u_n(0)=\sum_{j=1}^{J}g_n^j[e^{it_n^j\Delta}\phi^j]+W_n^J
\end{align}
and
\begin{align}\label{energy}
  \lim_{n\to\infty}E(u_n)=\sum_{j=1}^J\lim_{n\to\infty}E(e^{it_n^j\Delta}\phi^j)+\lim_{n\to\infty}E(W_n^J).
\end{align}
Since the kinetic energy decoupling \eqref{energy-decoupling} and energy is non-negative, then
$$\lim_{n\to\infty}E(e^{it_n^j\Delta}\phi^j)\leq E_c.$$

We claim that there only exists one profile, i.e., $J=1$. We will show this property by contradiction.
If more than one $\phi^j\neq0$, the energy decoupling implies that, for each $j\geq1$,
\begin{align}\label{E:B}
  \|\phi^j\|_{\dot{H}^1}<\|W\|_{\dot{H}^1}\ \text{and}\ \lim_{n\to\infty}E(e^{it_n^j\Delta}\phi^j)<E_c.
\end{align}
Fixing $j$, if $\big|\frac{x_n^j}{\lambda_n^j}\big|\to \infty $ up to a sub-sequence, by Lemma \ref{N-profile}, there exists a global solution $v_n^j$ to \eqref{INLC} with initial data $g_n^j[e^{it_n^j\Delta}\phi^j]$ and $\|v_n^j\|_{S(\R)}<\infty$. If $\big|\frac{x_n^j}{\lambda_n^j}\big|$ tends to $x_0\in (-\infty,\infty)$, we can always assume $x_n^j\equiv0$ in this case. For the setting $x_n^j\equiv0$, following the argument of \cite{KM}, we can construct the nonlinear profile $v^j$ associate to $(\phi^j,\{t_n^j\})$ such that
\begin{equation}\label{waveoperator}
\|v^j(t_n^j)-e^{it_n^j\Delta}\phi^j\|_{\dot{H}^1}\to 0,\ \ n\to\infty.
\end{equation}

Using this, Lemma \ref{energy-tra}, energy conservation and \eqref{E:B}, we can get $\sup_{t\in I_{\max}^j}\|v^j(t)\|_{\dot H^1}<\|W\|_{\dot H^1}$ and $E(v^j(t))<E_c$ which implies $\|v^j\|_{S(I_{\max}^j)}<\infty$ and $I_{\max}^j=\R$.

For the setting $\ x_n^j\equiv0$,  we define the rescaled functions
$$v_n^j(t,x)=g_n^j(v^j(\frac{t}{(\lambda_n^j)^2}+t_n^j, x)).$$ Thus $u_n(0)$ can be write
\begin{equation}\label{newprofile}
u_n(0)=\sum_{j=1}^{J}v_n^j(0)+\widetilde{W}_n^J,
\end{equation}
where
$$
\widetilde{W}_n=\sum_{j=1}^{J}\left[g_n^j[e^{it_n^j\Delta}\phi^j]-v_n^j(0)\right]+W_n^J.
$$
Note that, using the Strichartz estimate and the relations \eqref{vanishing-condition} - \eqref{waveoperator} we obtain
\begin{equation}\label{secondcondition}
 \lim_{J\to J^*}\lim_{n\to\infty}\|e^{it\Delta}\widetilde{W}_n^J\|_{S^1(\R)}=0.
\end{equation}

Let  $u_n^{J}$ denote by
\begin{align}
u_n^J(t)=\sum_{j=1}^{J}v_n^j(t)
\end{align}
and $F(z)=(I_{\alpha}\ast|\cdot|^{-b}|z|^p)|z|^{p-2}z$, thus $u_n^J(t)$ solves
$$i\partial_tu_n^J+\Delta u_n^J=-|x|^{-b}F(u_n^J)+e_n,$$
where
$$e_n=|x|^{-b}\left[F(\sum_{j=1}^Jv_n^j)-\sum_{j=1}^JF(v_n^j)\right].$$
We will use the perturbation theory to get a contradiction if we have proved the following three conditions:
 \begin{align}
    &\limsup_{J\to J^*}\limsup_{n\to\infty}\|u_n^J\|_{S^1(\R)}\lesssim1,\\
    &\limsup_{J\to J^*}\limsup_{n\to\infty}\|e^{it\Delta}(u_{0,n}-u_{n}^J(0))\|_{S^1(\R)}=0\\\label{Profile-main}
    &\limsup_{J\to J^*}\limsup_{n\to\infty}\|\nabla[(i\partial_t+\Delta)u_n^J+|x|^{-b}(I_{\alpha}\ast|\cdot|^{-b}|u_n^J|^p)|u_n^J|^{p-2}u_n^J]
    \|_{N(\R)}=0.
  \end{align}
  We first consider the first term, noting that
\begin{align*}
\Big|\sum_{j=1}^Jv_n^j\Big|^2=\sum_{j=1}^J|v_n^j|^2+\sum_{j\neq k}v_n^jv_n^k,
\end{align*}
then taking the $L_{t}^{\frac{q_0}{2}}L_x^{\frac{r_1}{2}}$ norm in both side, we can get
\begin{align}\label{sombounded}
\|u^J_n\|^2_{S^1(\R)}=\Big\|(\sum_{j=1}^{J}v_n^j)^2\Big\|_{L_{t}^{\frac{q_0}{2}}L_x^{\frac{r_1}{2}}}\leq
\sum_{j=1}^{J}\|v_n^j\|_{S^1(\R)}^2+\sum_{j\neq k}\|v_n^jv_n^k\|_{L_{t}^{\frac{q_0}{2}}L_x^{\frac{r_1}{2}}}.
\end{align}
The properties \eqref{smoothing} and \eqref{ZJ} imply that
$$\lim_{n\to\infty}\|v_n^jv_n^k\|_{L_{t}^{\frac{q_0}{2}}L_x^{\frac{r_1}{2}}}=0.$$
On the other hand, using profile decomposition we have
$$
\sum_{j=1}^J\|\phi^j\|_{\dot H^1}^2+\|W_n^J\|_{\dot H^1}^2=\|u_n(0)\|_{\dot H^1}^2+o_n(1)<\|W\|^2_{\dot H^1},
$$
which implies that $\sum_{j=1}^\infty\|\phi^j\|_{\dot H^1}^2\leq C_0$. Then we can choose $J_0\in \mathbb{N}$ large enough such that
$$\sum_{j=J_0}^\infty\|\phi^j\|_{\dot H^1}^2\leq \frac{\delta}{2},$$ for $\delta>0$ sufficiently small. Hence, fix $J>J_0$ and using \eqref{waveoperator} one has
$$
\sum_{j=J_0}^J\|v_n^j(0)\|_{\dot H^1}^2\leq \sum_{j=J_0}^J\left(\|v_n^j(0)-g_n^j[e^{it_n^j\Delta}\phi^j]\|_{\dot H^1}^2 +\|g_n^j[e^{it_n^j\Delta}\phi^j]\|_{\dot H^1}^2\right)\leq \delta.
$$
Therefore, by the Small Data Theory (Proposition \ref{LWP})
$$
\sum_{j=J_0}^J\|v_n^j\|_{S^1(\R)}^2\leq \delta,
$$
Applying \eqref{sombounded} and the last inequality we conlcude that $\|u^J_n\|_{S^1(\R)}$ is bounded (independent of J).

Now, we turn to prove the second term. Note that
$$u_{0,n}-u_n^J(0)=\widetilde{W}_n^J,$$
and by \eqref{secondcondition} we obtain the desired result.

 Finally, we turn to estimate \eqref{Profile-main}. Noting that
 \begin{align}\label{Reminder-1}
   (i\partial_t+\Delta)u_{n}^J+|x|^{-b}F(u_n^J)=|x|^{-b}[F(\sum_{j=1}^Jv_n^j)-\sum_{j=1}^JF(v_n^j)].
 \end{align}
 By directly computation, we have
 \begin{align*}
  \eqref{Reminder-1}=&|x|^{-b}\left(I_\alpha\ast|\cdot|^{-b}\left[\Big|\sum_{j=1}^{J}v_n^j\Big|^p-
  \sum_{j=1}^{J}|v_n^j|^p\right]\right)\Big|\sum_{j=1}^{J}v_n^j\Big|^{p-2}\sum_{j=1}^{J}v_n^j\\
  &+\sum_{j=1}^j|x|^{-b}(I_\alpha\ast|\cdot|^{-b}|v_n^j|^p)\Big|\sum_{j=1}^{J}v_n^j\Big|^{p-2}\sum_{j=1}^{J}v_n^j-\sum_{j=1}^{J}|x|^{-b}
  (I_\alpha\ast|\cdot|^{-b}|v_n^j|^p)|v_n^j|^{p-2}v_n^j\\
  =&\sum_{j=1}^J\Big(|x|^{-b}(I_\alpha\ast|\cdot|^{-b}|v_n^j|^p)\Big[\Big|\sum_{j=1}^{J}v_n^j\Big|^{p-2}\sum_{j=1}^{J}v_n^j-
  |v_n^j|^{p-2}v_n^j\Big]\Big)\\
  &+\sum_{j=1}^J|x|^{-b}\Big(I_\alpha\ast|\cdot|^{-b}\left[\Big|\sum_{j=1}^{J}v_n^j\Big|^p-
  \sum_{j=1}^{J}|v_n^j|^p\right]\Big)\Big|\sum_{j=1}^{J}v_n^j\Big|^{p-2}\sum_{j=1}^{J}v_n^j\\
  =&A+B.
 \end{align*}
 We only estimate the term $A$, the term $B$ is similar. A direct computation shows that
 \begin{align*}
   |\nabla A|\lesssim& |x|^{-b}(I_\alpha\ast|\cdot|^{-b}|v_n^k|^p)\left|\Big|\sum_{j=1}^{J}v_n^j\Big|^{p-2}\sum_{j=1}^{J}\nabla v_n^j-
  |v_n^k|^{p-2}\nabla v_n^k\right|\\
  &+|x|^{-b-1}(I_\alpha\ast|\cdot|^{-b}|v_n^k|^p)\left|\Big|\sum_{j=1}^{J}v_n^j\Big|^{p-2}
  \sum_{j=1}^{J}v_n^j-
  |v_n^k|^{p-2}v_n^k\right|\\
  &+|x|^{-b}(I_\alpha\ast|\cdot|^{-b}|v_n^k|^{p-1}\nabla v_n^k)\left|\Big|\sum_{j=1}^{J}v_n^j\Big|^{p-2}
  \sum_{j=1}^{J}v_n^j-
  |v_n^k|^{p-2}v_n^k\right|\\
  &+|x|^{-b}(I_\alpha\ast|\cdot|^{-b-1}|v_n^k|^p)\left|\Big|\sum_{j=1}^{J}v_n^j\Big|^{p-2}
  \sum_{j=1}^{J}v_n^j-
  |v_n^k|^{p-2}v_n^k\right|\\
  =&:I+II+III+IV,
 \end{align*}
where we omit the summation notation about indicator $k$. We only consider the first term $I$, since the other terms are similar. Applying Lemma \ref{HLS}, it suffices to show
 \begin{align}\label{ZJ-1}
  &\Big\|(I_\alpha\ast|\cdot|^{-b}|v_n^k|^p)|x|^{-b}|v_n^i||v_n^j|^{p-3}\nabla v_{n}^k\Big\|_{L_{t}^{q_0'}L_x^{r_0'}}\to 0, \ i\neq k\\\label{ZJ-2}
  &\Big\|(I_\alpha\ast|\cdot|^{-b}|v_n^k|^p)|x|^{-b}|\nabla v_n^i||v_n^j|^{p-2}\Big\|_{L_{t}^{q_0'}L_x^{r_0'}}\to0,\ \ i\neq k
\end{align}
as $n\to\infty$. We first estimate \eqref{ZJ-1}. For $b>1$, we have
\begin{align*}
 \Big\|(I_\alpha\ast|\cdot|^{-b}|v_n^k|^p)|x|^{-b}|v_n^i||v_n^j|^{p-3}\nabla v_{n}^k\Big\|_{L_{t}^{q_0'}L_x^{r_0'}}\lesssim \Big\||x|^{-1}v_n^i\nabla v_n^{k}\Big\|_{L_t^{\frac{q_0}2}L_x^{\frac{r_0}2}},
\end{align*}
by the orthogonality conditions \eqref{ZJ} and \eqref{smoothing}, then \eqref{ZJ-1} holds. If $b\leq1$, we use the following inequality
 \begin{align*}
\Big\|(I_\alpha\ast|\cdot|^{-b}|v_n^k|^p)|x|^{-b}|v_n^i||v_n^j|^{p-3}\nabla v_{n}^k\Big\|_{L_{t}^{q_0'}L_x^{r_0'}}\lesssim \Big\||x|^{-b}|v_n^i|^b\nabla v_n^{k}\Big\|_{L_t^{\frac{q_0}{1+b}}L_x^{\frac{r_0}{1+b}}},
\end{align*}
 we can obtain the same result.

We turn to prove \eqref{ZJ-2}. If $i\neq j$, the proof is the same as above. It suffices to show the case $i=j$. H\"older's inequality leads to
\begin{align*}
  \Big\|(I_\alpha\ast|\cdot|^{-b}|v_n^k|^p)|x|^{-b}|\nabla v_n^i||v_n^j|^{p-2}\Big\|_{L_{t}^{q_0'}L_x^{r_0'}}\lesssim \Big\|(I_\alpha\ast|\cdot|^{-b}|v_n^k|^p)\nabla v_n^i\Big\|_{L_t^{q}L_x^r},
\end{align*}
where $(q,r)$ satisfying
$$1-\frac1{r_0}=\frac1r+\frac{b}{r_0}+\frac{p-2-b}{r_1};$$
and
$$1-\frac1{q_0}=\frac1q+\frac{p-2}{q_0}.$$
By density and \eqref{smoothing}, we may assume $v_n^{k}=T_n^k\phi$ and $v_n^k=T_n^i\varphi$ where $\phi,\varphi\in C_c^\infty(\R\times\R^3)$. Using scaling and transformation, we obtain
\begin{align*}
 &\Big\|(I_\alpha\ast|\cdot|^{-b}|T_n^k\phi|^p)\nabla T_n^i\varphi\Big\|_{L_t^{q}L_x^r}\\
 \lesssim&(\lambda_n^k)^{-1}(\lambda_n^k)^{-\frac{p-1}2}\Big\|(I_\alpha\ast|\cdot|^{-b}T_n^k|\phi|^p)T_n^i\nabla \varphi\Big\|_{L_t^{q}L_x^r}\\
 \lesssim&(\lambda_n^k)^{-b-1}(\lambda_n^k)^{-\frac{p-1}2}\Big\|\Big(I_\alpha\ast T_n^k(|\cdot+x_n^k/\lambda_n^k|^{-b}|\phi|^p)\Big)T_n^i\nabla \varphi\Big\|_{L_t^{q}L_x^r}\\
 \lesssim&(\lambda_n^k)^{\alpha-b-1}(\lambda_n^k)^{-\frac{p-1}2}\Big\|T_n^k\Big(I_\alpha\ast (|\cdot+x_n^k/\lambda_n^k|^{-b}|\phi|^p)\Big)T_n^i\nabla \varphi\Big\|_{L_t^{q}L_x^r}\\
 \lesssim&(\lambda_n^k)^{\alpha-b-2+\frac2q+\frac3r}(\lambda_n^k)^{-\frac{p-1}2}\Big\|\Big(I_\alpha\ast (|\cdot+x_n^k/\lambda_n^k|^{-b}|\phi|^p)\Big)(T_n^k)^{-1}T_n^i\nabla \varphi\Big\|_{L_t^{q}L_x^r}.
\end{align*}
Since the choice of $(q,r)$, we have
\begin{align*}
 \Big \|(I_\alpha\ast|\cdot|^{-b}|T_n^k\phi|^p)\nabla T_n^i\varphi\Big\|_{L_t^{q}L_x^r}\lesssim&\Big\|\Big(I_\alpha\ast (|\cdot+x_n^k/\lambda_n^k|^{-b}|\phi|^p)\Big)(T_n^k)^{-1}T_n^i\nabla \varphi\Big\|_{L_t^{q}L_x^r}.
\end{align*}
H\"older's inequality and Hardy--Littlewood--Sobolev's inequality yield
$$I_\alpha\ast (|\cdot+x_n^k/\lambda_n^k|^{-b}|\phi|^p)\in L_t^{q_2}L_x^{r_2}(\R\times\R^3),$$
where
$$\frac1q=\frac1{q_2}+\frac1{q_1};\ \ \frac1{r}=\frac1{r_2}+\frac1{r_0}.$$
If $x_n\equiv0$, using the density of $C_c^\infty(\R\times\R^3)$ in space $L_t^{q_2}L_x^{r_2}(\R\times\R^3)$, thus there exist sequence $\{\varphi_n\}\subset C_c^\infty(\R\times\R^3)$ such that
$$\varphi_n\to I_\alpha\ast (|\cdot|^{-b}|\phi|^p)\ \textnormal{strongly in}\  L_t^{q_2}L_x^{r_2}(\R\times\R^3) .$$
If $\Big|\frac{x_n^k}{\lambda_n^k}\Big|\to\infty$, as $\phi\in C_c^\infty(\R\times\R^3)$, when $n$ large enough we have
$$I_\alpha\ast(|\cdot+x_n^k/\lambda_n^k|^{-b}|\phi|^p)\lesssim I_\alpha\ast (|\cdot|^{-b}|\phi|^p),$$
thus we can use $C_c^\infty$ function replacing it. It suffices to show that, for any $\phi_0$,
$$\|\phi_0(T_n^k)^{-1}T_n^i\nabla \varphi\|_{L_t^{q}L_x^r}\to 0$$
as $n\to\infty$. Using the support of $\phi_0,\varphi$, the result holds. For more detail, please see \cite{Vi}.

Up to now, we have shown only one profile, $\phi^1\neq0$ and $\phi^j=0$ for $j\geq2$.

\ We claim that $x_n^1 \equiv 0$. Otherwise, if $|\frac{x_n^1}{\lambda_n^1}|\rightarrow \infty$, then, by Proposition \ref{N-profile}, for all $n$ sufficiently large, there exists a global solution $v_n$ to \eqref{INLC} satisfying
$$
v_n(0)=\phi_n=g_n[e^{it_n\Delta}\phi] \;\;\textnormal{and}\;\; \|v_n\|_{S(\R)}\lesssim 1.
$$
Thus, by the long time perturbation, we show that $\|u_n\|_{S(\R)}<\infty$, which is a contradiction.

Now, we begin to prove the existence of a critical solution. As
$$E(e^{it_n^1\Delta}\phi^1)\to E_c$$
and by \eqref{newprofile} we deduce $E(v^1_n(0,x))\to E_c$, and so $E(\widetilde{W}_n^J)\to0$ which implies that
$$\widetilde{W}_n^J\to 0,\ \textnormal{strongly in} \ \dot{H}^1.$$
Then let $u_c=v^1$ (nonlinear profile of $(\phi^1,\{t_n^j\})$) which satisfies
$$\|u_c\|_{S(I_{\max})}=\infty$$
and $E(u_c)=E_c$. Thus, we complete the proof.
\end{proof}

\section{Preclusion of compact solutions}

In this section, we will remove the critical solution by the compactness of minimal energy blow-up solution and Morawetz estimate. Proposition \ref{N-profile} ultimately allows us to extend the construction of minimal energy blowup solutions from the radial to the non-radial setting; moreover, it guarantees that the compact solutions we ultimately construct remain localized near the origin, which facilitates the use of the localized virial argument.

Following the argument of Proposition \ref{MEBS}, we can get the compactness for minimal energy blow-up solution (see \cite{KM,KV} for more details).
\begin{proposition}\label{precompact}
  Let $u_c$ be the critical solution to \eqref{INLC}. There exists a frequency scale function $N:I_{\max}\to(0,\infty)$ such that
  $$\{N(t)^{-\frac12}u(t,N(t)^{-1}x):t\in I_{\max}\}$$
  is precompact in $\dot{H}^1(\R^3)$. In addition, we may assume $N(t)\geq1.$
\end{proposition}
\subsection{Finite-time blow-up}

The properties of compactness of $u_c$ imply the reduced Duhamel formula. Together with conservation of mass, we can preclude the blow-up setting.
\begin{proposition}[Reduced Duhamel formula \cite{KV}]\label{R-Duhamel}
  For $t\in [0,T_{\max})$, the following holds as a weak limit in $\dot{H}^1(\R^3)$:
  $$u(t)=i\lim_{T\to T_{\max}}\int_{t}^{T}e^{i(t-s)\Delta}(I_\alpha\ast|\cdot|^{-b}|u|^p)|x|^{-b}|u|^{p-2}u(s)ds.$$
\end{proposition}

\begin{proposition}
  If $T_{\max}<\infty,$ then $u\equiv0.$
\end{proposition}
\begin{proof}
  We suppose that $T_{\max}<\infty$. Using Strichartz, Hardy--Littlewood--Sobolev's inequality, Bernstein's inequality, and Hardy's inequality, we have
  \begin{align*}
    \|P_{N}u(t)\|_{L_x^2}&\lesssim \|P_N[(I_{\alpha}\ast|\cdot|^{-b}|u|^p)|x|^{-b}|u|^{p-2}u]\|_{L_t^1L_x^2}\\
    &\lesssim N^{3(\frac5{6}-\frac12)}|T_{\max}-t|\|(I_{\alpha}\ast|\cdot|^{-b}|u|^p)|x|^{-b}|u|^{p-2}u\|_{L_t^\infty L_x^{\frac65}}\\
    &\lesssim N|T_{\max}-t|\|\nabla u\|_{L_t^\infty L_x^2}^{2p-1}.
  \end{align*}
  Hence, Bernstein's inequality again implies that
  \begin{align*}
    \|u(t)\|_{L_x^2}&\leq\|P_Nu(t)\|_{L_x^2}+\|(1-P_N)u(t)\|_{L_x^2}\\
   &\lesssim  N|T_{\max}-t|+N^{-1}.
  \end{align*}
  By conservation of mass, we deduce $\|u\|_{L^2}=0$ and hence $u\equiv0.$
\end{proof}

\subsection{Soliton-like case}
In this section we assume $T_{\max}=\infty$ and get a contradiction by a Morawetz estimate.
\begin{lemma}[Morawetz identity \cite{SX}]\label{M-identity}
Let $a:\R^N\rightarrow\R$ be a smooth weight. Define
$$M(t)=2{\rm Im}\int_{\R^N}\bar{u}\nabla u\cdot\nabla adx.$$
Then
\begin{align*}
 \frac{d}{dt}M_a(t)=&\int_{\R^N}(-\Delta\Delta a)|u|^2dx+4\int_{\R^N}a_{jk}{\rm Re}(\pa_j\bar{u}\pa_ku)dx\\
 &-\bigg(2-\frac{4}{p}\bigg)\int_{\R^N}\Delta a|x|^{-b}|u|^p(I_\alpha\ast|\cdot|^{-b}|u|^p)dx\\
 &-\frac{4b}p\int_{\R^N}\nabla a\cdot\frac{x}{|x|^2}(I_\alpha\ast|\cdot|^{-b}|u|^p)|x|^{-b}|u|^pdx\\
 &-\frac{2\mathcal{K}(N-\alpha)}{p}\int_{\R^N}\int_{\R^N}(\nabla a(x)-\nabla a(y))\cdot\frac{x-y}{|x-y|^{N-\alpha+2}}|y|^{-b}|u|^p|x|^{-b}|u|^pdydx,
\end{align*}
where subscripts denote partial derivatives and repeated indices are summed.
\end{lemma}

\begin{lemma}[Tightness \cite{GM,KV}]
  Let $\epsilon>0$ and $p\geq b$.  Then there exists $R=R(\epsilon)$ sufficiently large so that
  \begin{align*}
    \sup_{t\in[0,\infty)}\int_{|x|>R}|\nabla u(t,x)|^2+|x|^{-2}|u(t,x)|^2+(I_\alpha\ast|\cdot|^{-b}|u|^p)|x|^{-b}|u|^p(t,x)dx<\epsilon.
  \end{align*}
\end{lemma}

\begin{proposition}
  If $T_{\max}=\infty$, then $u\equiv0.$
\end{proposition}
\begin{proof}
Let weight $a(x)$ denote by
\begin{align*}
a(x)=
  \begin{cases}
    |x|^2\ \ for\ \ |x|\leq R\\
    CR^2\ \ for\ \ |x|>2R,
  \end{cases}
\end{align*}
for some $C>1$. In this intermediate region, we can impose
$$|D^j a|\lesssim R^{2-j}\ \ for\ \ R<|x|\leq 2R.$$
On the other hand, Lemma \ref{M-identity} yields
  \begin{align}
    \frac{d}{dt}M(t)=&8\int_{\R^3}|\nabla u|^2-(I_\alpha\ast|\cdot|^b|u|^p)|x|^{-b}|u|^pdx\\
    &+\mathcal{O}\left(\int_{|x|>R}|\nabla u|^2+|x|^{-2}|u|^2+(I_\alpha\ast|\cdot|^b|u|^p)|x|^{-b}|u|^pdx\right).
  \end{align}
By Lemma \ref{energy-tra}, we have
$$E(u)\lesssim \frac{R^2}{\delta T}E(u)+\epsilon.$$
Choosing $T$ sufficiently large, we obtain $E(u)\lesssim \epsilon$. Then we deduce $E(u)\equiv0$ which implies $u\equiv0.$
 \end{proof}

\begin{proof}[Proof of Theorem \ref{theorem}] As in Section $4$, if Theorem \ref{theorem} fails, then there exits  a critical solution $u_c$ to \eqref{INLC}. From section $5$, we can conclude that $u_c\equiv0$, we derive a contradiction to $E(u_c)>0$. Thus we complete the proof of Theorem \ref{theorem}.

\end{proof}

We end this section by showing the corollaries.
\begin{proof}[Proof of Corollary \ref{Homog. setting}] We observe that Lemma \ref{Non-e1} also holds for $b=0$. We use the same admissible pairs with $b=0$. Thus, the results of well-posedness and stability hold in this case. Moreover, the other previous results used in the proof of Theorem \ref{theorem} also hold for this case. The proof is very similar to the inhomogeneous model. On the other hand, if the initial data is radial, then in the linear decomposition result we have $x_n^j\equiv 0$. Hence, the proof of this corollary follows exactly from Theorem \ref{theorem}. More precisely, when we constructed the critical solution (for the inhomogeneous case) we show that $x_n \equiv 0$. The same result as in the case $b\neq 0$. Therefore the proof is complete.
\end{proof}

\begin{proof}[Proof of Corollary \ref{def case}] The proof of this result also follows from Theorem \ref{theorem}. However, there are a few changes that we need to explain regarding the critical solution (Proposition \ref{MEBS}). In the focusing case, we have $E(u_0) < c$ and $\|u(0)\|_{\dot{H}^1}<\|W\|_{\dot{H}^1}$, On the other hand, in the defocusing case, it satisfies $E(u) < c$. It is important to note that we have previously shown that $E_c = E(u_c)$. For the defocusing case, we claim that $E_c = \infty$. In order to prove this, let's assume by contradiction that $E_c < \infty$. Since all the energies are positive and comparable with the kinetic energy, then the result follows similarly to the focusing case.
\end{proof}

\begin{center}

\end{center}
 \end{document}